%% file: modularQ.tex
\documentclass{amsart}
\usepackage[mathscr]{eucal}
\usepackage{mystyle}

\begin{document}
\title{Positivity and periodicity of $Q$-systems in the WZW fusion ring}
\author{Chul-hee Lee}
\address{School of Mathematics and Physics, The University of Queensland, Brisbane QLD 4072, Australia}
\email{c.lee1@uq.edu.au}
\thanks{This work was supported by the Max Planck Institute for Mathematics and the Australian Research Council.}
\date{February 2013, last modified on \today}
\begin{abstract}
\noindent
We study properties of solutions of $Q$-systems in the WZW fusion ring obtained by the Kirillov-Reshetikhin modules. We make a conjecture about their positivity and periodicity and give a proof of it in some cases. We also construct a positive solution of the level $k$ restricted $Q$-system of classical types in the fusion rings. As an application, we prove some conjectures of Kirillov and Kuniba-Nakanishi-Suzuki on the level $k$ restricted $Q$-systems.
\end{abstract}
\maketitle
\input{modularQcontents}
\bibliographystyle{amsalpha}
\bibliography{modularQ}
\end{document}

%% file: modularQcontents.tex
\section{Introduction}
Let $\mathfrak{g}$ be a simple Lie algebra over the complex numbers and $k\geq 1$ be an integer. In his work \cite{springerlink:10.1007/BF01840426} on dilogarithm identities associated $\mathfrak{g}$ and $k$, Kirillov claimed, without proof, that the character $Q_{m}^{(a)}$ of a certain finite-dimensional representation 
$\operatorname{res} W_m^{(a)}$ of $\mathfrak{g}$ vanishes if it is evaluated at the element 
$\rho/(k+h^{\vee})\in \mathfrak{h}^{*}$ when $m=t_a k+1$, 
where $h^{\vee}$ is the dual Coxeter number and $\rho$ is the Weyl vector of $\mathfrak{g}$.
Here $a$ denotes a vertex in the Dynkin diagram of $\mathfrak{g}$ and $t_a$ is the ratio of the norm of a long root to that of the simple root $\alpha_a$. This fact was necessary in order to obtain a finite set of conjecturally positive numbers $Q_{m}^{(a)}(\frac{\rho}{k+h^{\vee}})$ for $m=0, 1,\cdots, t_a k$, which was then used to construct the arguments for the Rogers dilogarithm function to formulate certain conjectural dilogarithm identities. 

Kirillov's claim on positivity and vanishing of $Q_{m}^{(a)}(\frac{\rho}{k+h^{\vee}})$ depending on $m$ could be easily settled when $\mathfrak{g}$ is of type $A$, but it was not quite obvious in other cases. After Kirillov's work, Kuniba, Nakanishi and Suzuki subsequently studied the problem in the more general setting and conjectured that similar phenomenon happens for some other elements of $\mathfrak{h}^{*}$, the dual of a Cartan subalgebra and added more conjectures to the original conjecture in the series of their works \cite{Kuniba1992, Kuniba1993, MR1304818,1751-8121-44-10-103001}.

Although the conjectural dilogarithm identity has been proven in \cite{springerlink:10.1007/BF01840426,MR2804544, MR3029995,MR3029994} for all types, the problem of identifying the positive arguments in the identity with the above specialization of characters had remained open for many years without any essential progress. It was quite recent that a proof in type $D$ was obtained in \cite{Lee2012} using the affine Weyl group symmetry in the specializations of the character. A proof in some cases of type $E$ was obtained in \cite{MR3282650}. 

In this paper, we show that the entire problem can be reformulated as the problem of finding the image of the Kirillov-Reshetikhin (KR) modules over $U_q(\hat{\mathfrak{g}})$ into the WZW fusion ring $\operatorname{Fus}_{k}(\mathfrak{g})$ under the composition map
$$
\operatorname{Rep} U_q(\hat{\mathfrak{g}})\xrightarrow{\operatorname{res}} \operatorname{Rep}U_q(\mathfrak{g})\xrightarrow{\beta_k} \operatorname{Fus}_{k}(\mathfrak{g})
$$
where $\operatorname{Rep} U_q(\hat{\mathfrak{g}})$ or $\operatorname{Rep}U_q(\mathfrak{g})$ denotes the Grothendieck ring of the category of finite-dimensional (type 1) representations of $U_q(\hat{\mathfrak{g}})$ or $U_q(\mathfrak{g})$, respectively; see subsections \ref{WZWf} and \ref{subsec:QKR}. These two ring homomorphisms $\operatorname{res}$ and $\beta_k$ had been topics of intensive study, but only separately. For example, to find the image of KR modules and their tensor products under $\operatorname{res}$, the Bethe ansatz has been applied to representation theory \cite{Kirillov1990,MR1745263}. The surjective homomorphism $\beta_k$ plays an important role in the structure theory of $\operatorname{Fus}_{k}(\mathfrak{g})$. Although the idea that one can view the relations among the characters of the KR modules as the one in $\operatorname{Fus}_{k}(\mathfrak{g})$ appeared in \cite{MR1228370}, this idea had not been pushed further.

The main results of this paper can be described as follows. With the above viewpoint and the setup, we prove some conjectures of Kirillov and Kuniba, Nakanishi and Suzuki in all classical types as an application of Theorem \ref{catQ}. Especially, when $\mathfrak{g}$ is of types $A,B,C$ and $D$, we have
$\mathcal{D}_{m}^{(a)}>0$ for $0\le m \le t_a k$ and $\mathcal{D}^{(a)}_{t_a k+1}=0$ where $\mathcal{D}_{m}^{(a)}=Q_{m}^{(a)}(\frac{\rho}{k+h^{\vee}})$ and $Q_{m}^{(a)}$ denotes the character of the KR module $\operatorname{res} W_m^{(a)}$; see Theorem \ref{qdimthm} and Corollary \ref{KNSconj} for more general and precise statements. Although some old questions are now answered, the new problem of determining the image of all KR modules in $\operatorname{Fus}_{k}(\mathfrak{g})$ is not fully solved here. We propose it as Conjecture \ref{pconj}.

The outline of this paper is as follows. In Section \ref{background}, we give a brief review on the notions of the modular $S$-matrix, the WZW fusion ring, $Q$-systems and KR modules. In Section \ref{pp}, we propose the conjecture about the positivity and periodicity of solutions of the $Q$-system in the WZW fusion ring obtained from the KR modules and prove it in some simple cases. At the end of the section, we define the notion of level $k$ restricted $Q$-systems. In Section \ref{sec:possol}, we study the solution of the level restricted $Q$-systems in the WZW fusion ring and especially present positive solutions of the level restricted $Q$-systems of all classical types. In Section \ref{sec:appsfusion}, we study complex solutions of level restricted $Q$-systems. In particular, as an application of our results in the previous section, we prove the conjecture of Kirillov and Kuniba-Nakanishi-Suzuki on the positivity and the other level truncation properties of $Q_{m}^{(a)}(\frac{\rho}{k+h^{\vee}})$ for all classical types.

\subsection*{Acknowledgements}
The author would like to thank the Max Planck Institute for Mathematics for its hospitality and support, where most of this work was done.

\section{Notation and background}\label{background}
In this section, we fix notation related to a simple Lie algebra $\mathfrak{g}$ and recall some results about the modular $S$-matrix and the WZW fusion ring. For a reference, see \cite[Chapter 14 and Chapter 16]{philippe1997conformal}. For $Q$-systems and KR modules in Section \ref{subsec:QKR}, see \cite[Section 13]{1751-8121-44-10-103001} and the references therein.
\subsection{Notation}
Throughout the paper, we will use the following notation :
\begin{itemize}

\item $I=\{1,2,\cdots, r\}$ the index set of the Dynkin diagram of $\mathfrak{g}$ as in Table \ref{tab:Dynkin}
\item $\hat{I}=\{0\}\cup I$ as in Table \ref{tab:Dynkin}
\item $C=(C_{ab})_{a,b\in I}$ the Cartan matrix
\item $\mathfrak{g}$ the finite-dimensional complex simple Lie algebra associated to $C$
\item $\mathfrak{h}$ a Cartan subalgebra of $\mathfrak{g}$
\item $\mathfrak{h}^{*}$ the dual space of $\mathfrak{h}$
\item $\alpha_{i}, i\in I$ the simple roots
\item $\omega_{i}, i\in I$ the fundamental weights
\item $\alpha_{i}^{\vee}, i\in I$ the simple coroots
\item $\Pi=\{\alpha_{i}: i\in I\}$ the simple system
\item $\Delta$ the root system
\item $\Delta_{+}$ the set of positive roots
\item $Q=\oplus_{i\in I}\mathbb{Z} \alpha_i$ the root lattice 
\item $P=\oplus_{i\in I}\mathbb{Z} \omega_i$ the weight lattice
\item $P_{+}=\{\lambda \in P : \lambda =\sum_{i=1}^{r}\lambda_{i}\omega_{i},\,\lambda_{i}\ge 0\}$
\item $Q^{\vee}=\oplus_{i\in I}\mathbb{Z} \alpha_i^{\vee}$ the coroot lattice
\item $\theta$ the highest root $$\theta=\sum_{i=1}^{r}a_i\alpha_{i}=\sum_{i=1}^{r}c_i\alpha_{i}^{\vee}$$
\item $a_0=1$ and $a_i, i\in I$ the marks
\item $c_0=1$ and $c_i, i\in I$ the comarks
\item $(\cdot|\cdot)$ standard symmetric bilinear form on $\mathfrak{h^{*}}$ satisfying $(\alpha_{i}^{\vee}|\alpha_j)=C_{ij}$, $(\omega_{i}|\alpha_{i}^{\vee})=\delta_{ij}$ and the normalization condition $(\theta|\theta)=2$
\item $h=\sum_{i=0}^{r}a_i$ the Coxeter number
\item $h^{\vee}=\sum_{i=0}^{r}c_i$ the dual Coxeter number
\item  $e^{\lambda}, \lambda\in P$ the function $e^{\lambda} : \mathfrak{h}^{*} \to \mathbb{C}$ defined by 
$$\mu \mapsto \exp{2\pi i (\lambda|\mu)}$$
\item $\hat{P}=\oplus_{i\in \hat{I}}\mathbb{Z} \hat{\omega}_i$ the affine weight lattice
\item $\hat{P}^k=\{\hat{\lambda} \in \hat{P}:\hat{\lambda}=\sum_{i=0}^{r}\lambda_{i}\hat{\omega}_{i}\text{ such that}\, \sum_{i=0}^{r}c_{i}\lambda_{i}=k\}$
\item $\hat{P}^{k}_{+}=\{\hat{\lambda} \in \hat{P}^k:\hat{\lambda}=\sum_{i=0}^{r}\lambda_{i}\hat{\omega}_{i},\,\lambda_{i}\ge 0\}$
\item $\pi :\hat{P}\to P$ defined by
\begin{equation}\label{pi}
\pi(\sum_{i=0}^{r}\lambda_{i}\hat{\omega}_{i}):=\sum_{i=1}^{r}\lambda_{i}\omega_{i}.
\end{equation}
\item $\alpha_0=\alpha_0^{\vee}=-\theta$
\item $\hat{\alpha}_j=\sum_{i=0}^{r} (\alpha_j|\alpha_i^{\vee}) \hat{\omega}_{i}$, $j\in \hat{I}$
\item $s_i,i\in I$  the fundamental reflections on $P$ defined by 
$$s_i \lambda=\lambda -(\alpha_i^{\vee}|\lambda)\alpha_i$$
\item $s_i, i\in \hat{I}$ the fundamental reflections on $\hat{P}$ defined linearly by 
$$s_i \hat{\omega}_j=\hat{\omega}_j -\delta_{ij}\hat{\alpha}_i$$
where $\delta_{ij}$ denotes the Kronecker delta
\item $W$ the finite Weyl group generated by $s_i, i\in I$ (acting on $\mathfrak{h^{*}}$)
\item $\widehat{W}$ the affine Weyl group generated by $s_i, i \in \hat{I}$ (acting on $\hat{P}$)
\item $\ell(w)$ the length of $w\in W$, or $\widehat{W}$
\item $\rho=\sum_{i=1}^{r}\omega_{i}\in P$ the Weyl vector
\item $\hat{\rho}=\sum_{i=0}^{r}\hat{\omega}_{i}\in \hat{P}$ the affine Weyl vector 
\item $w\cdot \hat{\lambda}$ for $w\in \widehat{W}$ and $\hat{\lambda}\in\hat{P}$, the shifted affine Weyl group action $$w\cdot \hat{\lambda}:=w(\hat{\lambda}+\hat{\rho})-\hat{\rho}$$
\item $H=\{(a,m):a\in I, m\in \mathbb{Z}_{\ge 0}\}$
\item $H_k=\{(a,m)\in H:a\in I, 0\le m \le t_a k\}$
\item $\mathring{H}_k=\{(a,m)\in H:a\in I, 1\le m \le t_a k-1\}$
\item $\lfloor x \rfloor$ the greatest integer not exceeding $x$
\item $t_a=\frac{2}{(\alpha_a|\alpha_a)},\, a\in I$
\end{itemize}
We give a remark on the bilinear form $(\cdot|\cdot)$. Whenever we have an affine weight $\hat{\lambda}\in \hat{P}$, $(\hat{\lambda}|\cdot)$ is defined to be $(\pi(\hat{\lambda})|\cdot)$. For a given $\lambda\in P$ and an integer $k\ge 1$, we can find a unique element $\hat{\lambda}\in \hat{P}^{k}$ such that $\pi(\hat{\lambda})=\lambda$. We call $\hat{\lambda}$ the \textit{level $k$ affinization} of $\lambda$. 
\unitlength=.95pt
\begin{table}
\caption{The Dynkin diagrams and the extended Dynkin diagrams (reproduced from \cite{MR1745263})}
\label{tab:Dynkin}
\begin{tabular}[t]{rl}
$A_r$:&
\begin{picture}(106,20)(-5,-5)
\multiput( 0,0)(20,0){2}{\circle{6}}
\multiput(80,0)(20,0){2}{\circle{6}}
\multiput( 3,0)(20,0){2}{\line(1,0){14}}
\multiput(63,0)(20,0){2}{\line(1,0){14}}
\multiput(39,0)(4,0){6}{\line(1,0){2}}
\put(0,-5){\makebox(0,0)[t]{$1$}}
\put(20,-5){\makebox(0,0)[t]{$2$}}
\put(80,-5){\makebox(0,0)[t]{$r\!\! -\!\! 1$}}
\put(100,-5){\makebox(0,0)[t]{$r$}}
\end{picture}
\\
&
\\
$B_r$:&
\begin{picture}(106,20)(-5,-5)
\multiput( 0,0)(20,0){2}{\circle{6}}
\multiput(80,0)(20,0){2}{\circle{6}}
\multiput( 3,0)(20,0){2}{\line(1,0){14}}
\multiput(63,0)(20,0){1}{\line(1,0){14}}
\multiput(82.85,-1)(0,2){2}{\line(1,0){14.3}}
\multiput(39,0)(4,0){6}{\line(1,0){2}}
\put(90,0){\makebox(0,0){$>$}}
\put(0,-5){\makebox(0,0)[t]{$1$}}
\put(20,-5){\makebox(0,0)[t]{$2$}}
\put(80,-5){\makebox(0,0)[t]{$r\!\! -\!\! 1$}}
\put(100,-5){\makebox(0,0)[t]{$r$}}
\end{picture}
\\
&
\\
$C_r$:&
\begin{picture}(106,20)(-5,-5)
\multiput( 0,0)(20,0){2}{\circle{6}}
\multiput(80,0)(20,0){2}{\circle{6}}
\multiput( 3,0)(20,0){2}{\line(1,0){14}}
\multiput(63,0)(20,0){1}{\line(1,0){14}}
\multiput(82.85,-1)(0,2){2}{\line(1,0){14.3}}
\multiput(39,0)(4,0){6}{\line(1,0){2}}
\put(90,0){\makebox(0,0){$<$}}
\put(0,-5){\makebox(0,0)[t]{$1$}}
\put(20,-5){\makebox(0,0)[t]{$2$}}
\put(80,-5){\makebox(0,0)[t]{$r\!\! -\!\! 1$}}
\put(100,-5){\makebox(0,0)[t]{$r$}}
\end{picture}
\\
&
\\
$D_r$:&
\begin{picture}(106,40)(-5,-5)
\multiput( 0,0)(20,0){2}{\circle{6}}
\multiput(80,0)(20,0){2}{\circle{6}}
\put(80,20){\circle{6}}
\multiput( 3,0)(20,0){2}{\line(1,0){14}}
\multiput(63,0)(20,0){2}{\line(1,0){14}}
\multiput(39,0)(4,0){6}{\line(1,0){2}}
\put(80,3){\line(0,1){14}}
\put(0,-5){\makebox(0,0)[t]{$1$}}
\put(20,-5){\makebox(0,0)[t]{$2$}}
\put(80,-5){\makebox(0,0)[t]{$r\!\! - \!\! 2$}}
\put(103,-5){\makebox(0,0)[t]{$r\!\! -\!\! 1$}}
\put(85,20){\makebox(0,0)[l]{$r$}}
\end{picture}
\\
&
\\
$E_6$:&
\begin{picture}(86,40)(-5,-5)
\multiput(0,0)(20,0){5}{\circle{6}}
\put(40,20){\circle{6}}
\multiput(3,0)(20,0){4}{\line(1,0){14}}
\put(40, 3){\line(0,1){14}}
\put( 0,-5){\makebox(0,0)[t]{$1$}}
\put(20,-5){\makebox(0,0)[t]{$2$}}
\put(40,-5){\makebox(0,0)[t]{$3$}}
\put(60,-5){\makebox(0,0)[t]{$4$}}
\put(80,-5){\makebox(0,0)[t]{$5$}}
\put(45,20){\makebox(0,0)[l]{$6$}}
\end{picture}
\\
&
\\
$E_7$:&
\begin{picture}(106,40)(-5,-5)
\multiput(0,0)(20,0){6}{\circle{6}}
\put(40,20){\circle{6}}
\multiput(3,0)(20,0){5}{\line(1,0){14}}
\put(40, 3){\line(0,1){14}}
\put( 0,-5){\makebox(0,0)[t]{$1$}}
\put(20,-5){\makebox(0,0)[t]{$2$}}
\put(40,-5){\makebox(0,0)[t]{$3$}}
\put(60,-5){\makebox(0,0)[t]{$4$}}
\put(80,-5){\makebox(0,0)[t]{$5$}}
\put(100,-5){\makebox(0,0)[t]{$6$}}
\put(45,20){\makebox(0,0)[l]{$7$}}
\end{picture}
\\
&
\\
$E_8$:&
\begin{picture}(126,40)(-5,-5)
\multiput(0,0)(20,0){7}{\circle{6}}
\put(80,20){\circle{6}}
\multiput(3,0)(20,0){6}{\line(1,0){14}}
\put(80, 3){\line(0,1){14}}
\put( 0,-5){\makebox(0,0)[t]{$1$}}
\put(20,-5){\makebox(0,0)[t]{$2$}}
\put(40,-5){\makebox(0,0)[t]{$3$}}
\put(60,-5){\makebox(0,0)[t]{$4$}}
\put(80,-5){\makebox(0,0)[t]{$5$}}
\put(100,-5){\makebox(0,0)[t]{$6$}}
\put(120,-5){\makebox(0,0)[t]{$7$}}
\put(85,20){\makebox(0,0)[l]{$8$}}
\end{picture}
\\
&
\\
$F_4$:&
\begin{picture}(66,20)(-5,-5)
\multiput( 0,0)(20,0){4}{\circle{6}}
\multiput( 3,0)(40,0){2}{\line(1,0){14}}
\multiput(22.85,-1)(0,2){2}{\line(1,0){14.3}}
\put(30,0){\makebox(0,0){$>$}}
\put(0,-5){\makebox(0,0)[t]{$1$}}
\put(20,-5){\makebox(0,0)[t]{$2$}}
\put(40,-5){\makebox(0,0)[t]{$3$}}
\put(60,-5){\makebox(0,0)[t]{$4$}}
\end{picture}
\\
&
\\
$G_2$:&
\begin{picture}(26,20)(-5,-5)
\multiput( 0, 0)(20,0){2}{\circle{6}}
\multiput(2.68,-1.5)(0,3){2}{\line(1,0){14.68}}
\put( 3, 0){\line(1,0){14}}
\put( 0,-5){\makebox(0,0)[t]{$1$}}
\put(20,-5){\makebox(0,0)[t]{$2$}}
\put(10, 0){\makebox(0,0){$>$}}
\end{picture}
\\
&
\\
\end{tabular}
\begin{tabular}[t]{rl}
$A_1^{(1)}$:&
\begin{picture}(26,20)(-5,-5)
\multiput( 0,0)(20,0){2}{\circle{6}}
\multiput(2.85,-1)(0,2){2}{\line(1,0){14.3}}
\put(0,-5){\makebox(0,0)[t]{$0$}}
\put(20,-5){\makebox(0,0)[t]{$1$}}
\put( 6, 0){\makebox(0,0){$<$}}
\put(14, 0){\makebox(0,0){$>$}}
\end{picture}
\\
&
\\
\begin{minipage}[b]{4em}
\begin{flushright}
$A_r^{(1)}$:\\$(r \ge 2)$
\end{flushright}
\end{minipage}&
\begin{picture}(106,40)(-5,-5)
\multiput( 0,0)(20,0){2}{\circle{6}}
\multiput(80,0)(20,0){2}{\circle{6}}
\put(50,20){\circle{6}}
\multiput( 3,0)(20,0){2}{\line(1,0){14}}
\multiput(63,0)(20,0){2}{\line(1,0){14}}
\multiput(39,0)(4,0){6}{\line(1,0){2}}
\put(2.78543,1.1142){\line(5,2){44.429}}
\put(52.78543,18.8858){\line(5,-2){44.429}}
\put(0,-5){\makebox(0,0)[t]{$1$}}
\put(20,-5){\makebox(0,0)[t]{$2$}}
\put(80,-5){\makebox(0,0)[t]{$r\!\! -\!\! 1$}}
\put(100,-5){\makebox(0,0)[t]{$r$}}
\put(55,20){\makebox(0,0)[lb]{$0$}}
\end{picture}
\\
&
\\
\begin{minipage}[b]{4em}
\begin{flushright}
$B_r^{(1)}$:\\$(r \ge 3)$
\end{flushright}
\end{minipage}&
\begin{picture}(126,40)(-5,-5)
\multiput( 0,0)(20,0){3}{\circle{6}}
\multiput(100,0)(20,0){2}{\circle{6}}
\put(20,20){\circle{6}}
\multiput( 3,0)(20,0){3}{\line(1,0){14}}
\multiput(83,0)(20,0){1}{\line(1,0){14}}
\put(20,3){\line(0,1){14}}
\multiput(102.85,-1)(0,2){2}{\line(1,0){14.3}} 
\multiput(59,0)(4,0){6}{\line(1,0){2}} 
\put(110,0){\makebox(0,0){$>$}}
\put(0,-5){\makebox(0,0)[t]{$1$}}
\put(20,-5){\makebox(0,0)[t]{$2$}}
\put(40,-5){\makebox(0,0)[t]{$3$}}
\put(100,-5){\makebox(0,0)[t]{$r\!\! -\!\! 1$}}
\put(120,-5){\makebox(0,0)[t]{$r$}}
\put(25,20){\makebox(0,0)[l]{$0$}}
\end{picture}
\\
&
\\
\begin{minipage}[b]{4em}
\begin{flushright}
$C_r^{(1)}$:\\$(r \ge 2)$
\end{flushright}
\end{minipage}&
\begin{picture}(126,20)(-5,-5)
\multiput( 0,0)(20,0){3}{\circle{6}}
\multiput(100,0)(20,0){2}{\circle{6}}
\multiput(23,0)(20,0){2}{\line(1,0){14}}
\put(83,0){\line(1,0){14}}
\multiput( 2.85,-1)(0,2){2}{\line(1,0){14.3}} 
\multiput(102.85,-1)(0,2){2}{\line(1,0){14.3}} 
\multiput(59,0)(4,0){6}{\line(1,0){2}} 
\put(10,0){\makebox(0,0){$>$}}
\put(110,0){\makebox(0,0){$<$}}
\put(0,-5){\makebox(0,0)[t]{$0$}}
\put(20,-5){\makebox(0,0)[t]{$1$}}
\put(40,-5){\makebox(0,0)[t]{$2$}}
\put(100,-5){\makebox(0,0)[t]{$r\!\! -\!\! 1$}}
\put(120,-5){\makebox(0,0)[t]{$r$}}
\end{picture}
\\
&
\\
\begin{minipage}[b]{4em}
\begin{flushright}
$D_r^{(1)}$:\\$(r \ge 4)$
\end{flushright}
\end{minipage}&
\begin{picture}(106,40)(-5,-5)
\multiput( 0,0)(20,0){2}{\circle{6}}
\multiput(80,0)(20,0){2}{\circle{6}}
\multiput(20,20)(60,0){2}{\circle{6}}
\multiput( 3,0)(20,0){2}{\line(1,0){14}}
\multiput(63,0)(20,0){2}{\line(1,0){14}}
\multiput(39,0)(4,0){6}{\line(1,0){2}}
\multiput(20,3)(60,0){2}{\line(0,1){14}}
\put(0,-5){\makebox(0,0)[t]{$1$}}
\put(20,-5){\makebox(0,0)[t]{$2$}}
\put(80,-5){\makebox(0,0)[t]{$r\!\! - \!\! 2$}}
\put(103,-5){\makebox(0,0)[t]{$r\!\! -\!\! 1$}}
\put(25,20){\makebox(0,0)[l]{$0$}}
\put(85,20){\makebox(0,0)[l]{$r$}}
\end{picture}
\\
&
\\
$E_6^{(1)}$:&
\begin{picture}(86,60)(-5,-5)
\multiput(0,0)(20,0){5}{\circle{6}}
\multiput(40,20)(0,20){2}{\circle{6}}
\multiput(3,0)(20,0){4}{\line(1,0){14}}
\multiput(40, 3)(0,20){2}{\line(0,1){14}}
\put( 0,-5){\makebox(0,0)[t]{$1$}}
\put(20,-5){\makebox(0,0)[t]{$2$}}
\put(40,-5){\makebox(0,0)[t]{$3$}}
\put(60,-5){\makebox(0,0)[t]{$4$}}
\put(80,-5){\makebox(0,0)[t]{$5$}}
\put(45,20){\makebox(0,0)[l]{$6$}}
\put(45,40){\makebox(0,0)[l]{$0$}}
\end{picture}
\\
&
\\
$E_7^{(1)}$:&
\begin{picture}(126,40)(-5,-5)
\multiput(0,0)(20,0){7}{\circle{6}}
\put(60,20){\circle{6}}
\multiput(3,0)(20,0){6}{\line(1,0){14}}
\put(60, 3){\line(0,1){14}}
\put( 0,-5){\makebox(0,0)[t]{$0$}}
\put(20,-5){\makebox(0,0)[t]{$1$}}
\put(40,-5){\makebox(0,0)[t]{$2$}}
\put(60,-5){\makebox(0,0)[t]{$3$}}
\put(80,-5){\makebox(0,0)[t]{$4$}}
\put(100,-5){\makebox(0,0)[t]{$5$}}
\put(120,-5){\makebox(0,0)[t]{$6$}}
\put(65,20){\makebox(0,0)[l]{$7$}}
\end{picture}
\\
&
\\
$E_8^{(1)}$:&
\begin{picture}(146,40)(-5,-5)
\multiput(0,0)(20,0){8}{\circle{6}}
\put(100,20){\circle{6}}
\multiput(3,0)(20,0){7}{\line(1,0){14}}
\put(100, 3){\line(0,1){14}}
\put( 0,-5){\makebox(0,0)[t]{$0$}}
\put(20,-5){\makebox(0,0)[t]{$1$}}
\put(40,-5){\makebox(0,0)[t]{$2$}}
\put(60,-5){\makebox(0,0)[t]{$3$}}
\put(80,-5){\makebox(0,0)[t]{$4$}}
\put(100,-5){\makebox(0,0)[t]{$5$}}
\put(120,-5){\makebox(0,0)[t]{$6$}}
\put(140,-5){\makebox(0,0)[t]{$7$}}
\put(105,20){\makebox(0,0)[l]{$8$}}
\end{picture}
\\
&
\\
$F_4^{(1)}$:&
\begin{picture}(86,20)(-5,-5)
\multiput( 0,0)(20,0){5}{\circle{6}}
\multiput( 3,0)(20,0){2}{\line(1,0){14}}
\multiput(42.85,-1)(0,2){2}{\line(1,0){14.3}} 
\put(63,0){\line(1,0){14}}
\put(50,0){\makebox(0,0){$>$}}
\put( 0,-5){\makebox(0,0)[t]{$0$}}
\put(20,-5){\makebox(0,0)[t]{$1$}}
\put(40,-5){\makebox(0,0)[t]{$2$}}
\put(60,-5){\makebox(0,0)[t]{$3$}}
\put(80,-5){\makebox(0,0)[t]{$4$}}
\end{picture}
\\
&
\\
$G_2^{(1)}$:&
\begin{picture}(46,20)(-5,-5)
\multiput( 0,0)(20,0){3}{\circle{6}}
\multiput( 3,0)(20,0){2}{\line(1,0){14}}
\multiput(22.68,-1.5)(0,3){2}{\line(1,0){14.68}}
\put( 0,-5){\makebox(0,0)[t]{$0$}}
\put(20,-5){\makebox(0,0)[t]{$1$}}
\put(40,-5){\makebox(0,0)[t]{$2$}}
\put(30,0){\makebox(0,0){$>$}}
\end{picture}
\\
&
\\
\end{tabular}
\end{table}
\unitlength=1pt

\subsection{Modular S-matrix}
Let $k\ge 1$ be an integer. For a pair of weights $\lambda, \mu \in P$, we consider the following quantity
\begin{equation}
S_{\hat{\lambda},\hat{\mu}}=\frac{i^{|\Delta_{+}|}}{\sqrt{|P/Q^{\vee}|(k+h^{\vee})^r}} \sum_{w\in W} (-1)^{\ell(w)}\exp \left(-{\frac{2\pi i ( w(\lambda+\rho)|\mu+\rho)}{k+h^{\vee}}}\right)
\label{modularS}
\end{equation}
where $\hat{\lambda}$ and $\hat{\mu}$ are the level $k$ affinizations of $\lambda$ and $\mu$, respectively. Note that $S_{\hat{\lambda},\hat{\mu}}=S_{\hat{\mu},\hat{\lambda}}$. 

We call the matrix $S=(S_{\hat{\lambda}, \hat{\mu}})_{\hat{\lambda},\hat{\mu}\in \hat{P}_{+}^{k}}$ the \textit{modular $S$-matrix}. It appears when one tries to describe the modular transformation properties of the characters of the integrable highest weight representations of level $k$ over the affine Kac-Moody algebras \cite{MR750341}. Let us review some of their properties.

For the modular $S$-matrix, we have the following orthogonality relation
\begin{equation}\label{orthoS}
SS^{\dagger}=I_n
\end{equation}
where $S^{\dagger}$ denotes the transpose of the complex conjugate of $S$ and $I_n$ is the identity matrix of size $n=|\hat{P}_{+}^{k}|$. In other words, $S$ is a unitary matrix.  

The modular $S$-matrix satisfies various symmetries which will be heavily used in the following sections. The shifted action of the affine Weyl group gives 
\begin{equation}\label{shiftWS}
S_{w\cdot \hat{\lambda}, \hat{\mu}}=(-1)^{\ell(w)}S_{\hat{\lambda},\hat{\mu}}
\end{equation}
for $w\in \widehat{W}$. Let $O(\hat{\mathfrak{g}})$ be the outer automorphism group of $\hat{\mathfrak{g}}$ consisting of some diagram automorphisms of the extended Dynkin diagram. See Table \ref{autext} for a description.
We use the standard notation for permutations and cycles to represent its elements. For example, for the permutation
$$ 
\tau=\left(
\begin{array}{ccccc}
0 &1 &\cdots & {r-1} &{r} \\
\tau(0) &\tau(1) &\cdots & \tau(r-1) & \tau(r) \\
\end{array}
\right),
$$
we define $\tau\hat{\omega}_a:=\hat{\omega}_{\tau(a)}$ for each $a\in \hat{I}$. This also gives their action on $\hat{P}$. The cycle $(i \quad j \quad \cdots \quad k)$ sends $\hat{\omega}_i$ to $\hat{\omega}_j$ and $\hat{\omega}_k$ to $\hat{\omega}_i$. Note that $\tau\in O(\hat{\mathfrak{g}})$ can be uniquely determined by the image $\tau\hat{\omega}_0$ of $\hat{\omega}_0$. For $\tau \in O(\hat{\mathfrak{g}})$, we have
\begin{equation}\label{outAutS}
S_{\tau \hat{\lambda}, \hat{\mu}}=S_{\hat{\lambda}, \hat{\mu}}e^{-2\pi i (\tau \hat{\omega}_0|\mu)}.
\end{equation}
Thus the action of $O(\hat{\mathfrak{g}})$ results in the multiplication of a certain root of unity on the entries of $S$-matrix. 

\begin{table}
\caption{Generators of the outer automorphism group $O(\hat{\mathfrak{g}})$}
\begin{center}
\label{autext}
\begin{tabular}{c|c|c}
$\mathfrak{g}$ & $O(\hat{\mathfrak{g}})$ & generators \\
\hline
$A_r$ & $\mathbb{Z}_{r+1}$ & 
$(0\quad r\quad r-1\quad\cdots \quad 1)$
\\
\hline
$B_r$ & $\mathbb{Z}_2$ &
$
(0\quad 1)
$
\\
\hline
$C_r$ & $\mathbb{Z}_2$ & 
$
\left(
\begin{array}{ccccc}
0 &1 &\cdots & {r-1} &{r}\\
{r} &{r-1} &\cdots & {1} &{0}
\end{array}
\right)
$
\\
\hline
$D_{r}$, $r$ even & $\mathbb{Z}_2\times \mathbb{Z}_2$ &
$(0\quad 1)(r-1\quad r)$,

$
\left(
\begin{array}{cccccc}
0 &1 &2 &\cdots & {r-1} &{r} \\
{r} & {r-1} & {r-2}& \cdots & {1} &{0}
\end{array}
\right)
$
\\
\hline
$D_{r}$ $r$ odd & $\mathbb{Z}_4$ & 
$
\left(
\begin{array}{cccccc}
0 &1 &2 &\cdots & {r-1} &{r} \\
{r-1} & {r} & {r-2}& \cdots & {1} &{0}
\end{array}
\right)
$
\\
\hline
$E_{6}$ & $\mathbb{Z}_3$ & 
$
\left(
\begin{array}{ccccccc}
0 &1 & 2 & 3 & 4 & 5 & {6} \\
{1} &{5} &{4} &{3} &{6} &{0} &{2}
\end{array}
\right)
$
\\
\hline
$E_{7}$ & $\mathbb{Z}_2$ &
$
\left(
\begin{array}{cccccccc}
0 &1 & 2 & 3 & 4 & 5 & {6}& {7} \\
{6} &{5} &{4} &{3} &{2} &{1} &{0} &{7}
\end{array}
\right)
$
\end{tabular}
\end{center}
\end{table}
For a finite weight $\lambda\in P$, let us consider the conjugate weight $\lambda^{*}:=-w_0 \lambda \in P$ where $w_0$ is the longest element of the finite Weyl group $W$. The conjugate weight $\hat{\lambda}^{*}$ of the affine weight $\hat{\lambda}$ is defined to be the level $k$ affinization  of $\lambda^{*}\in P$. This conjugation can be regarded as a diagram automorphism of the extended Dynkin diagram preserving the vertex $0$. For concreteness, we give their action in Table \ref{longB} using the permutation notation. For $\hat{\lambda},\hat{\mu}\in \hat{P}_{+}^{k}$, we have
\begin{equation}\label{longestconj}
S_{\hat{\lambda}^{*},\hat{\mu}}=S^{*}_{\hat{\lambda},\hat{\mu}}.
\end{equation}

Let us also recall some properties of the affine Weyl group $\widehat{W}$. For $\lambda\in P$ and its image $\lambda' \in P$ under the reflection through the affine hyperplane $H_{\alpha,n}=\{x\in \mathfrak{h_{\mathbb{R}}^{*}}:(\alpha|x)=n\}$ where $n$ is an integer, we can find an odd element $w\in \widehat{W}$ such that
\begin{equation}\label{refhyp}
\hat{\lambda'}=w\hat{\lambda}\in \hat{P}^n.
\end{equation}

For $\hat{\lambda}\in \hat{P}^{k}$, we define the \textit{quantum dimension} $\mathcal{D}_{\hat{\lambda}}$ by
\begin{equation}
\mathcal{D}_{\hat{\lambda}} : =
\frac{S_{\hat{\lambda}, \hat{0}}}{S_{\hat{0}, \hat{0}}}
=\frac{\prod_{\alpha\in \Delta_{+}}\sin \frac{\pi(\lambda+\rho|\alpha)}{k+h^{\vee}}}{\prod_{\alpha\in \Delta_{+}}\sin \frac{\pi (\rho|\alpha)}{k+h^{\vee}}} \label{quanProd}
\end{equation}
where the equality can be justified by the Weyl character formula.

\begin{prop}\label{pro:quantumD}
Let $\hat{\lambda}\in \hat{P}^{k}$. The following conditions are equivalent :
\begin{enumerate}[label=(\roman{*}), ref=(\roman{*})]

\item \label{itm:pro:quantumD1} $\mathcal{D}_{\hat{\lambda}}=0$
\item \label{itm:pro:quantumD2} there exists $w\in \widehat{W}$ of odd signature such that $w\cdot \hat{\lambda}=\hat{\lambda}$
\item \label{itm:pro:quantumD3} $S_{\hat{\lambda}, \hat{\mu}}=0$ for all $\hat{\mu}\in \hat{P}_{+}^{k}$
\end{enumerate}
\end{prop}
\begin{proof}
Let us prove \ref{itm:pro:quantumD1}$\Rightarrow$\ref{itm:pro:quantumD2}. Assume that $\mathcal{D}_{\hat{\lambda}}=0$. From (\ref{quanProd}), we can find $\alpha\in \Delta$ such that $(\lambda+\rho|\alpha)$ is $n(k+h^{\vee})$ for some integer $n$. This implies that
$\lambda+\rho$ is fixed under the reflection through the affine hyperplane $H_{\alpha,n(k+h^{\vee})}=\{x\in \mathfrak{h_{\mathbb{R}}^{*}}:(\alpha|x)=n(k+h^{\vee})\}$. 
By (\ref{refhyp}), there exists $w\in \widehat{W}$ of odd signature such that $w\cdot \hat{\lambda}=\hat{\lambda}$.  \ref{itm:pro:quantumD2}$\Rightarrow$\ref{itm:pro:quantumD3} follows from (\ref{shiftWS}). \ref{itm:pro:quantumD3}$\Rightarrow$\ref{itm:pro:quantumD1} is obvious from (\ref{quanProd}).
\end{proof}

We also note that if $\mathcal{D}_{\hat{\lambda}}\neq 0$ for $\hat{\lambda}\in \hat{P}^{k}$, then we can find a unique element $\hat{\lambda}'\in \hat{P}_{+}^{k}$ such that
\begin{equation}\label{alcoverep}
\hat{\lambda}'=w\cdot \hat{\lambda}
\end{equation}
for some $w\in \widehat{W}$. If $\hat{\lambda}\in \hat{P}_{+}^{k}$, then $\mathcal{D}_{\hat{\lambda}}>0$ from (\ref{quanProd}) and thus 
\begin{equation}\label{Sneq0}
S_{\hat{\lambda},\hat{0}}=\mathcal{D}_{\hat{\lambda}}S_{\hat{0}, \hat{0}}\neq 0.
\end{equation}
\subsection{WZW fusion ring}\label{WZWf}
The \textit{WZW fusion ring} $\operatorname{Fus}_{k}(\mathfrak{g})$ is a free $\mathbb{Z}$-module equipped with the basis $\{V_{\hat{\lambda}} :\hat{\lambda} \in \hat{P}_{+}^{k}\}$ with a certain product structure on it, called the \textit{fusion product}, given by
$$V_{\hat{\lambda}}\cdot V_{\hat{\mu}}=\sum_{\hat{\nu}\in \hat{P}_{+}^{k}}N_{\hat{\lambda} \hat{\mu}}^{\hat{\nu}}V_{\hat{\nu}}$$
where $N_{\hat{\lambda} \hat{\mu}}^{\hat{\nu}}$ is the dimension of the conformal block $V_{\mathbb{P}^1}(\hat{\lambda},\hat{\mu},\hat{\nu}^{*})$ on $\mathbb{P}^1$
; see \cite{MR1360497, MR2499554} and the references therein. As a ring, it is commutative and associative with unity $V_{k \hat{\omega }_0}$. We call $N_{\hat{\lambda} \hat{\mu}}^{\hat{\nu}}$ the \textit{fusion coefficient}. The Verlinde formula 
\begin{equation}\label{Verlinde}
N_{\hat{\lambda} \hat{\mu}}^{\hat{\nu}}=\sum_{\hat{\omega} \in \hat{P}_{+}^{k}}\frac{S_{\hat{\lambda},\hat{\omega}}S_{\hat{\mu},\hat{\omega}}S^{*}_{\hat{\nu},\hat{\omega}}}{S_{\hat{0},\hat{\omega}}}, 
\end{equation}
relates the fusion coefficients with the modular $S$-matrix.

Recall that the Grothendieck ring $\operatorname{Rep}\mathfrak{g}$ of finite-dimensional representations of $\mathfrak{g}$ is a free $\mathbb{Z}$-module with the basis $\{V_{\omega} :\omega \in P_{+}\}$, where $V_{\omega}$ denotes the isomorphism class of an irreducible highest weight module of highest weight ${\omega}$. Here the product structure is given by the tensor product of the corresponding representations. There exists a surjective ring homomorphism $\beta_k : \operatorname{Rep}\mathfrak{g}\to \operatorname{Fus}_{k}(\mathfrak{g})$ defined by :
\begin{equation}
\beta_k(V_{\lambda}):=
\begin{cases} 
0 & \text{if } \mathcal{D}_{\hat{\lambda}}=0\\ 
(-1)^{\ell(w)}V_{\hat{\lambda}'} & \text{if } \mathcal{D}_{\hat{\lambda}}\neq 0
\end{cases}\label{Valcove}
\end{equation}
for $\lambda\in P_{+}$. Here $\hat{\lambda}' \in \hat{P}_{+}^{k}$ denotes the element in the fundamental alcove as in (\ref{alcoverep}); see \cite{MR1360497, MR2499554} for more on this map. In short, to find the image of $\beta_{k}$, we need to move the dominant integral weight in the fundamental Weyl chamber into the fundamental Weyl alcove by using the affine Weyl group and count the number of simple reflections necessary to achieve it.

\begin{table}
\caption{The diagram automorphism corresponding to $*$}
\begin{center}
\label{longB}
\begin{tabular}{c|c}
$\mathfrak{g}$ & $*$ 
\\
\hline
$A_r$ & 
$
\left(
\begin{array}{cccc}
 0 &  1 & \cdots &  r \\
 0 &  r & \cdots &  1 \\
\end{array}
\right)
$ 
\\
\hline
$D_{r}$, $r$ odd & $({r-1}\quad r)$
\\
\hline
$E_{6}$ & 
$
(1\quad 5)(2\quad 4)
$
\\
\hline
otherwise & trivial
\end{tabular}
\end{center}
\end{table}
There exists an involution $*:\operatorname{Fus}_{k}(\mathfrak{g})\to \operatorname{Fus}_{k}(\mathfrak{g})$ given by 
$$
V^{*}_{\hat{\omega}}:=V_{\hat{\omega}^{*}},
$$
which satisfies the following commutative diagram
\begin{center}
\begin{tikzcd}[row sep = 9ex, column sep = 20ex, ampersand replacement=\&]
\operatorname{Rep}\mathfrak{g}
    \arrow{r}{\beta_k} 
    \arrow[swap]{d}{*}
\& 
\operatorname{Fus}_{k}(\mathfrak{g}) 
    \arrow{d}{*}  \\
\operatorname{Rep} \mathfrak{g}
    \arrow[swap]{r}{\beta_k} 
\& 
\operatorname{Fus}_{k}(\mathfrak{g})
\end{tikzcd}
\end{center}
The involution $*:\operatorname{Rep}\mathfrak{g}\to \operatorname{Rep}\mathfrak{g}$ is given by $V_{\lambda}^{*}=V_{\lambda^{*}}$.
We can also define the action of $O(\hat{\mathfrak{g}})$ on $\operatorname{Fus}_{k}(\mathfrak{g})$ by
$$
\tau V_{\hat{\omega}}:=V_{\tau \hat{\omega}},\, \tau\in O(\hat{\mathfrak{g}}).
$$

\begin{definition}
Let $V=\sum_{\hat{\lambda}\in \hat{P}_{+}^{k}}Z_{\hat{\lambda}}V_{\hat{\lambda}}\in \operatorname{Fus}_{k}(\mathfrak{g})$. If $Z_{\hat{\lambda}}\ge 0$ for all $\hat{\lambda}\in \hat{P}_{+}^{k}$, then we call $V$ \textit{non-negative}. If $V$ is non-negative with at least one $Z_{\hat{\lambda}}$ nonzero, then we call $V$ \textit{positive}. We define \textit{non-positive} and \textit{negative} elements in a similar way.
\end{definition}

\begin{definition}
For $\hat{\lambda},\hat{\mu}\in \hat{P}_{+}^{k}$, we define the \textit{generalized quantum dimension} of $V_{\hat{\lambda}}\in \operatorname{Fus}_{k}(\mathfrak{g})$ by
\begin{equation}\label{gqdim}
\operatorname{qdim}_{\hat{\mu}}V_{\hat{\lambda}}:=\frac{S_{\hat{\lambda}, \hat{\mu}}}{S_{\hat{0},\hat{\mu} }}
\end{equation}
For $V\in \operatorname{Fus}_{k}(\mathfrak{g})$, we define  $\operatorname{qdim}_{\hat{\mu}}V$ by extending (\ref{gqdim}) linearly. 
For $\hat{\mu}=\hat{0}=k\hat{\omega}_0$, we will use the notation 
$$
\operatorname{qdim}_{\hat{0}}V:=\operatorname{qdim}V.
$$
\end{definition}
Let us summarize various properties of the generalized quantum dimensions, many of which follow from the properties of the modular $S$-matrix.

\begin{prop}\label{prop:qdimmu}
Let $\hat{\mu}\in \hat{P}_{+}^{k}$, $\tau\in O(\hat{\mathfrak{g}})$ and $V\in \operatorname{Fus}_{k}(\mathfrak{g})$. The following properties hold :
\begin{enumerate}[label=(\roman{*}), ref=(\roman{*})]

\item \label{itm:hom} $\operatorname{qdim}_{\hat{\mu}}: \operatorname{Fus}_{k}(\mathfrak{g})\to \mathbb{C}$ is a homomorphism,
\item \label{itm:pos} If $V$ is positive, then $\operatorname{qdim}V>0$,
\item \label{itm:cpxconj} $\operatorname{qdim}_{\hat{\mu}}V^{*}=(\operatorname{qdim}_{\hat{\mu}} V)^{*}$,
\item \label{itm:rtunity}$\operatorname{qdim}_{\hat{\mu}} V_{k(\tau \hat{\omega}_0)}=e^{-2\pi i (\tau \hat{\omega}_0|\mu)}$,
\item \label{itm5:prop:qdimmu}$\operatorname{qdim}_{\hat{\mu}}\tau V=(\operatorname{qdim}_{\hat{\mu}} V_{k(\tau \hat{\omega}_0)})(\operatorname{qdim}_{\hat{\mu}} V)$,
\item \label{itm4:prop:qdimmu} $V=0$ if and only if $\operatorname{qdim}_{\hat{\mu}}V=0$ for all $\hat{\mu}\in \hat{P}_{+}^{k}$.
\end{enumerate}
\end{prop}
\begin{proof}
\ref{itm:hom} follows from (\ref{orthoS}) and (\ref{Verlinde}); see \cite[Proposition 9.4]{MR1360497} also.
\ref{itm:pos} follows from (\ref{quanProd}). \ref{itm:cpxconj} is a consequence of (\ref{longestconj}). 
For \ref{itm:rtunity} and \ref{itm5:prop:qdimmu}, we can use (\ref{outAutS}).
If $V=\sum_{\hat{\lambda}\in \hat{P}_{+}^{k}}Z_{\hat{\lambda}}V_{\hat{\lambda}}$, then $$S_{\hat{0}, \hat{\mu}}\cdot \operatorname{qdim}_{\hat{\mu}}V=\sum_{\hat{\lambda}\in \hat{P}_{+}^{k}}Z_{\hat{\lambda}}S_{\hat{\lambda}, \hat{\mu}}.$$ Then \ref{itm4:prop:qdimmu} follows from (\ref{orthoS}).
\end{proof}

The above proposition implies that for $\tau \in O(\hat{\mathfrak{g}})$ and $V\in \operatorname{Fus}_{k}(\mathfrak{g})$,
\begin{equation}\label{autmult}
\tau V=V_{k(\tau \hat{\omega}_0)}\cdot V,
\end{equation}
which allows us to interpret the action of $O(\hat{\mathfrak{g}})$ in $\operatorname{Fus}_{k}(\mathfrak{g})$ as multiplications by certain invertible elements of $\operatorname{Fus}_{k}(\mathfrak{g})$.

\begin{definition}
For $V\in \operatorname{Fus}_{k}(\mathfrak{g})$, let $M_{V}=(m_{\hat{\lambda} \hat{\mu}})_{\hat{\lambda},\hat{\mu}\in \hat{P}_{+}^{k}}$ be the matrix defined by
$$
V\cdot V_{\hat{\lambda}}=\sum_{\hat{\mu}\in \hat{P}_{+}^{k}}m_{\hat{\lambda} \hat{\mu}}V_{\hat{\mu}}.
$$
We call $M_{V}$ the \textit{fusion matrix} of $V$. 
\end{definition}
Statements on positive elements in $\operatorname{Fus}_{k}(\mathfrak{g})$ can be translated into the corresponding ones on their fusion matrices as follows :
\begin{prop}\label{Perron-Frobenius}
Let $V\in \operatorname{Fus}_{k}(\mathfrak{g})$ be positive. Then  
\begin{enumerate}[label=(\roman{*}), ref=(\roman{*})]

\item $M_V$ is an integral matrix with non-negative entries,
\item $\operatorname{qdim}_{\hat{\mu}}V$ is an eigenvalue of $M_{V}$ with an eigenvector $(\operatorname{qdim}_{\hat{\mu}} V_{\hat{\lambda}})_{\hat{\lambda}\in \hat{P}_{+}^{k}}$,
\item $\operatorname{qdim} V$ is the Perron-Frobenius eigenvalue of $M_V$.
\end{enumerate}
\end{prop}
\begin{proof}
It follows from the fact that the fusion coefficient $N_{\hat{\lambda} \hat{\mu}}^{\hat{\nu}}$ is a non-negative integer and Proposition \ref{prop:qdimmu}.
\end{proof}

\subsection{$Q$-systems and Kirillov-Reshetikhin modules}\label{subsec:QKR}
\begin{definition}
Let $R$ be a commutative ring with unity. Let $\mathbf{Q}=\{Q^{(a)}_{m}\}_{(a,m)\in H}$ be a family of elements in $R$ satisfying
\begin{equation}\label{eq:qsys}
\left(Q^{(a)}_{m}\right)^2 = Q^{(a)}_{m-1}Q^{(a)}_{m+1} + 
\prod_{b: b \sim a} \prod_{j=0}^{-C_{a b}-1}
Q^{(b)}_{\lfloor\frac{C_{b a}m -j}{C_{a b}}\rfloor}, \quad (a,m)\in H.
\end{equation}
We call (\ref{eq:qsys}) the \textit{unrestricted $Q$-system of type $\mathfrak{g}$}
and $\mathbf{Q}$ a \textit{solution} of it. Here $b \sim a$ means $C_{ab}<0$. We assume $Q^{(a)}_{-1} =0$ for all $a\in I$.
\end{definition}
When the Dynkin diagram of $\mathfrak{g}$ is simply-laced, (\ref{eq:qsys}) is of the form
$$
\left(Q^{(a)}_{m}\right)^2=Q^{(a)}_{m-1}Q^{(a)}_{m+1}+\prod _{b:b \sim a} Q^{(b)}_{m}.
$$

Let $q$ be a non-zero complex number which is not a root of unity. The finite-dimensional irreducible representations of type 1 over the quantum affine algebra $U_{q}(\hat{\mathfrak{g}})$ are classified by the set of $I$-tuples $\mathbf{P}=(P_i)_{i\in I}$ of polynomials $P_i\in \mathbb{C}[z]$, with $P_i(0)=1$ as shown in  \cite{MR1357195} and they are called Drinfeld polynomials. The Grothendieck ring $\operatorname{Rep} U_q(\hat{\mathfrak{g}})$ of finite-dimensional representations of $U_{q}(\hat{\mathfrak{g}})$ is a free $\mathbb{Z}$-module with the basis $\{V_{\mathbf{P}}:\mathbf{P}=(P_i)_{i\in I},P_i\in \mathbb{C}[z],\,P_i(0)=1\}$, where $V_{\mathbf{P}}$ denotes the simple module associated to $\mathbf{P}$. It is a commutative ring; see, for example, \cite[Corollary 2]{frenkel1999q}.

There exists a special class of finite-dimensional irreducible modules of $U_{q}(\hat{\mathfrak{g}})$ called the Kirillov-Reshetikhin (KR) modules. A KR module is associated with Drinfeld polynomials $\mathbf{P}=(P_i)_{i\in I}$ of the form
$$
P_i(z) =
\begin{cases} 
 \prod _{s=1}^m \left(1- z u q_{a}^{2(s-1)}\right), & \text{if $i=a$}\\
 1, & \text{otherwise} \\ 
\end{cases}
$$
for some $(a,m)\in H$ and $u\in \mathbb{C}^{\times}$, where $q_{a} = q^{t/t_a}$ and $t=\max_{a\in I}t_a$ and is denoted by $W^{(a)}_{m,u}$. We can obtain a finite-dimensional $U_{q}(\mathfrak{g})$-module ${\rm res}\, W^{(a)}_{m,u}$ by restriction which allows us to drop the dependence on the spectral parameter $u$. From now on, we will drop $u$ when we write ${\rm res}\, W^{(a)}_{m,u}$. Although it is usual to write an isomorphism class of a finite-dimensional representation $V$ of $U_{q}(\hat{\mathfrak{g}})$ or $U_{q}(\mathfrak{g})$ as $[V]$, we will denote both of them by $V$ indiscriminately by abuse of notation as we will be mostly working with the Grothendieck ring in this paper.

As claimed in \cite{Kirillov1990} and proved in \cite{MR1993360, MR2254805}, the family 
$\{{\rm res}\, W^{(a)}_{m}\}_{(a,m)\in H}$ is a solution of (2.15) in $\operatorname{Rep}U_q(\mathfrak{g})$; Nakajima and Hernandez actually proved a stronger statement that the $q$-characters of KR modules satisfy the $T$-system whose specialization ignoring the spectral parameter by restriction is the $Q$-system.
In general, ${\rm res}\, W^{(a)}_m$ is not irreducible as a $U_{q}(\mathfrak{g})$-module and its decomposition into irreducibles gives
\begin{equation}\label{Qdecomp}
\operatorname{res} W^{(a)}_m=\sum_{\lambda \in P_{+}}Z(a, m,\lambda)V_{\lambda}
\end{equation}
for some non-negative integers $Z(a, m,\lambda)\in \mathbb{Z}$. In fact, $Z(a, m,m\omega_a)=1$ and $Z(a, m,\lambda)\neq 0$ only if $m\omega_a-\lambda$ can be written as a linear combination of $\alpha_1,\cdots,\alpha_r$ with non-negative coefficients.
See \cite[Appendix]{MR1745263} for a thorough treatment of this topic. We give an explicit description of (\ref{Qdecomp}) for all classical types in Appendix \ref{appQdec}. 

From what we will see later, it seems natural to separate and focus only on certain parts of a solution of $Q$-systems when it satisfies the boundary conditions $Q^{(a)}_{t_a k+1}=0$ for all $a\in I$. To deal with this situation, let us give the following definition.
\begin{definition}\label{levresQ}
Let $\mathbf{Q}=\{Q^{(a)}_{m}\}_{(a,m)\in H_k}$ be a family of elements in $R$ satisfying
\begin{equation}\label{eq:lresq}
\left(Q^{(a)}_{m}\right)^2 = Q^{(a)}_{m-1}Q^{(a)}_{m+1} + 
\prod_{b:b \sim a} \prod_{j=0}^{-C_{a b}-1}
Q^{(b)}_{\lfloor\frac{C_{b a}m - j}{C_{a b}}\rfloor}, \quad (a,m)\in H_k
\end{equation}
with $Q^{(a)}_{-1}=Q^{(a)}_{t_a k+1}=0$ for all $a\in I$. 
We call (\ref{eq:lresq}) \textit{the level $k$ restricted $Q$-system of type $\mathfrak{g}$} and $\mathbf{Q}$ a \textit{solution} of it.
\end{definition}

\section{Positivity and periodicity of $Q$-systems in $\operatorname{Fus}_{k}(\mathfrak{g})$}\label{pp}
\subsection{Main problem}
Now we can state the main problem in our study. For each $(a,m)\in H$, we want to determine the image $\beta_k(\operatorname{res}W^{(a)}_{m})$ of $W^{(a)}_{m}$ under the composition
$$
\operatorname{Rep} U_q(\hat{\mathfrak{g}})\xrightarrow{\operatorname{res}} \operatorname{Rep}U_q(\mathfrak{g}) \xrightarrow{\beta_k} \operatorname{Fus}_{k}(\mathfrak{g}),
$$
where we have used the known isomorphism between $\operatorname{Rep}U_q(\mathfrak{g})$ and $\operatorname{Rep}\mathfrak{g}$.

\begin{prop}\label{pr:betaim}
The family $\{\beta_k(\operatorname{res} W^{(a)}_{m})\}_{(a,m)\in H}$ is a solution of the unrestricted $Q$-system in $\operatorname{Fus}_{k}(\mathfrak{g})$. 
\end{prop}
\begin{proof}
This follows from the facts that $\{\operatorname{res} W^{(a)}_{m}\}_{(a,m)\in H}$ is a solution of the unrestricted $Q$-system and that $\beta_k$ is a ring homomorphism.
\end{proof}

Let us take a look at an example.
\begin{example}
Let $\mathfrak{g}=D_5$ and $k=4$. From (\ref{Domega1}), we can find
$$
\operatorname{res} W^{(2)}_{4}=V_{0}+V_{\omega_2}+V_{2\omega_2}+V_{3\omega_2}+V_{4\omega_2}.
$$

Let us compute $\beta_k(\operatorname{res} W^{(2)}_{4})$. For the weights $0,\omega_2,2\omega_2,3\omega_2$ and $4\omega_2$, the corresponding affine weights in $\hat{P}^{4}$ are 
$4\hat{\omega}_0, 2\hat{\omega}_0+\hat{\omega}_2, 2\hat{\omega}_2, -2\hat{\omega}_0+3\hat{\omega}_2$ and $-4\hat{\omega}_0+4\hat{\omega}_2$, respectively.
The shifted action of the affine Weyl group gives us
$$
\begin{aligned}
s_0\cdot (-2 \hat{\omega }_ 0+3 \hat{\omega }_ 2)&=2\hat{\omega}_2, \\ 
s_0\cdot (-4\hat{\omega}_0+4\hat{\omega}_2)&=2\hat{\omega}_0+\hat{\omega}_2.
\end{aligned}
$$
Thus 
$$
\begin{aligned}
\beta_k(\operatorname{res} W^{(2)}_{4})&=V_{4\hat{\omega}_0}+V_{2\hat{\omega}_0+\hat{\omega}_2}+V_{2\hat{\omega}_2}+(-1)V_{2\hat{\omega}_2}+(-1)V_{2\hat{\omega}_0+\hat{\omega}_2}\\
&=V_{4\hat{\omega}_0}
\end{aligned}
$$ from (\ref{Valcove}).
\end{example}
For $\hat{\mu}\in \hat{P}_{+}^{k}$, we will denote the generalized quantum dimension of $W^{(a)}_{m}$ by $$\mathcal{D}^{(a)}_{m,\hat{\mu}}:=\operatorname{qdim}_{\hat{\mu}}\beta_k(\operatorname{res} W^{(a)}_{m}).$$ Note that we get a solution
$$\{\mathcal{D}^{(a)}_{m,\hat{\mu}}\}_{(a,m)\in H}$$
of the unrestricted $Q$-system in $\mathbb{C}$ from Proposition \ref{prop:qdimmu}. For $\hat{\mu}=\hat{0}$, we will write $$\mathcal{D}^{(a)}_{m}:=\operatorname{qdim}\beta_k(\operatorname{res} W^{(a)}_{m}).$$ 

\subsection{Gauges of $Q$-systems}
To give a conjectural description of $\beta_k(\operatorname{res}W^{(a)}_{m})$ for all $(a,m)\in H$, we introduce a useful concept related $Q$-systems.  
\begin{definition}
If a family of elements $\mathbf{g}=\{g^{(a)}\}_{a\in I}$ in $R^{\times}$ satisfies
\begin{equation}\label{Qbdryeq}
\left(g^{(a)}\right)^2 =\prod_{b:b \sim a} \left(g^{(b)}\right)^{-C_{a b}} , \quad a\in I, 
\end{equation}
we call it a \textit{gauge} of the $Q$-system of type $\mathfrak{g}$. Note that (\ref{Qbdryeq}) can be rewritten as
$$
\prod_{b=1}^{r} \left(g^{(b)}\right)^{C_{a b}}=1.
$$
\end{definition}
Note that $\mathbf{1}=\{1\}_{a\in I}$ is a gauge. For two gauges $\mathbf{g}$ and $\mathbf{h}$, we can define the product $\mathbf{gh}$ as the component-wise product $\{g^{(a)}h^{(a)}\}_{a\in I}$, which is a gauge. For a gauge $\mathbf{g}=\{g^{(a)}\}_{a\in I}$, $\mathbf{g}^{-1}:=\{(g^{(a)})^{-1}\}_{a\in I}$ is also a gauge.
\begin{prop}
The set of gauges forms a group.
\end{prop}
\begin{proof}
This is trivial.
\end{proof}
For a gauge $\mathbf{g}=\{g^{(a)}\}_{a\in I}$ and a family $\mathbf{Q}=\{Q^{(a)}_{m}\}_{(a,m)\in H}$ in $R$, we let $\mathbf{g}\cdot \mathbf{Q}:=\{g^{(a)}Q^{(a)}_{m}\}_{(a,m)\in H}$. For $\mathbf{Q}=\{Q^{(a)}_{m}\}_{(a,m)\in H_k}$, we define $\mathbf{g}\cdot \mathbf{Q}$ in the same way. 
\begin{prop}\label{gauge}
Let $\mathbf{g}$ be a gauge of the $Q$-system. If $\mathbf{Q}$ is a solution of the unrestricted $Q$-system or the level $k$ restricted $Q$-system, then so is $\mathbf{g}\cdot\mathbf{Q}$. 

\end{prop}
\begin{proof}
This follows from (\ref{eq:qsys}), (\ref{eq:lresq}) and (\ref{Qbdryeq}) in a straightforward way.
\end{proof}

\begin{table}
\caption{$\hat{\tau}_a\in \hat{P}$ to describe a gauge}
\begin{center}
\label{taua}
\begin{tabular}{c|c}
$\mathfrak{g}$ & $\hat{\tau}_a $ \\
\hline
$A_r$ & $\hat{\omega}_a$ \\
\hline
$B_r$ & 
$\begin{cases} 
\hat{\omega}_1 & \text{if $a\equiv 1 \pmod 2$}\\ 
\hat{\omega}_0 & \text{if $a\equiv 0 \pmod 2$}
\end{cases}$
\\
\hline
$C_r$ & 
$\begin{cases} 
\hat{\omega}_0 & \text{if $1\le a \le r-1$}\\ 
\hat{\omega}_r & \text{if $a=r$}
\end{cases}$
\\
\hline
$D_r$ & 
$\begin{cases} 
\hat{\omega}_1 & \text{if $1\le a \le r-2$ and $a\equiv 1 \pmod 2$}\\ 
\hat{\omega}_0 & \text{if $1\le a \le r-2$ and $a\equiv 0 \pmod 2$}\\
\hat{\omega}_a & \text{if $a=r-1$ or $a=r$}
\end{cases}$
\\
\hline
$E_6$ & $\begin{array}{c|c|c|c|c|c|c}
a &1 & 2 & 3 & 4 & 5 & 6 \\
\hline
\hat{\tau}_a  & \hat{\omega}_1 & \hat{\omega}_5  & \hat{\omega}_0 & \hat{\omega}_1  & \hat{\omega}_5 & \hat{\omega}_0 
\end{array}$
\\
\hline
$E_7$ & $\begin{array}{c|c|c|c|c|c|c|c}
a& 1 & 2 & 3 & 4 & 5 & 6 & 7\\
\hline
\hat{\tau}_a & \hat{\omega}_0 & \hat{\omega}_0  & \hat{\omega}_0 & \hat{\omega}_6 & \hat{\omega}_0 & \hat{\omega}_6 & \hat{\omega}_6
\end{array}$
\\
\hline
$E_8,F_4,G_2$ & $\hat{\omega}_0$
\end{tabular}
\end{center}
\end{table}

\begin{lemma}\label{checkunitA}
Let $\{\hat{\tau}_{a}\}_{a\in I}$ be as in Table \ref{taua}. Then we have
\begin{equation}
\sum_{b\in I} C_{a b}\pi(\hat{\tau}_b) \in Q^{\vee}, \quad a\in I. \label{checkunitB}
\end{equation}
\begin{proof}
We list the conditions (\ref{checkunitB}) to be checked for each type explicitly. 

For type $A_r$, we see that $\sum_{b\in I} C_{a b} \omega_b=\alpha_a,\,a\in I$, which is in $Q=Q^{\vee}$.

For type $B_r$, we only have to check that $2\omega_1\in Q^{\vee}$. 

For type $C_r$, the only condition we need to check is $2\omega_r\in Q^{\vee}$.

For type $D_{r}$, $r$ odd, we have the following conditions
$$
\begin{cases} 
2\omega_{1}\in Q^{\vee}\\ 
\omega_{r-1}+\omega_{r}-2\omega_{1}\in Q^{\vee}\\
2\omega_{r-1}-\omega_{1}\in Q^{\vee} \\
2\omega_{r}-\omega_{1}\in Q^{\vee}
\end{cases}.
$$

For type $D_{r}$, $r$ even, we need to check
$$
\begin{cases} 
2\omega_{1}\in Q^{\vee}\\ 
\omega_{1}+\omega_{r-1}+\omega_{r}\in Q^{\vee}\\
2\omega_{r-1}\in Q^{\vee} \\
2\omega_{r}\in Q^{\vee}
\end{cases}.
$$

For type $E_{6}$, the conditions to be checked are given by
$$
\begin{cases} 
2\omega_{1}-\omega_{5} \in Q^{\vee} \\
2\omega_{5}-\omega_{1} \in Q^{\vee} \\
\omega_{1}+\omega_{5}\in Q^{\vee} \\
\end{cases}.
$$

For type $E_{7}$, we only have to check  $2{\omega}_6\in Q^{\vee}$.

These can all be verified by a straightforward calculation.
\end{proof}
\end{lemma}

\begin{prop}\label{boundarysol}
Let $\{\hat{\tau}_{a}\}_{a\in I}$ be as in Table \ref{taua}. Then $\{V_{k\hat{\tau}_a}\}_{a\in I}$ is a gauge of the $Q$-system in $\operatorname{Fus}_{k}(\mathfrak{g})$.
\end{prop}

\begin{proof}
For types $E_8, F_4$ and $G_2$, it is trivial since $V_{k\hat{\tau}_a}$ is the identity in $\operatorname{Fus}_{k}(\mathfrak{g})$ for any $a\in I$. 

By Proposition \ref{prop:qdimmu} \ref{itm4:prop:qdimmu}, we can show that $\{V_{k \hat{\tau}_a}\}_{a\in I}$ is a gauge of the $Q$-system by proving 
\begin{equation}\label{qbd}
\prod_{b\in I} \left(\operatorname{qdim}_{\hat{\mu}}V_{k \hat{\tau}_b}\right)^{C_{a b}}=1
\end{equation}
for any $a\in I$ and $\hat{\mu}\in\hat{P}_{+}^{k}$.
As
$$\operatorname{qdim}_{\hat{\mu}}V_{k \hat{\tau}_a}=e^{-2\pi i (\hat{\tau}_a | \mu)}$$
from Proposition \ref{prop:qdimmu} \ref{itm:rtunity}, (\ref{qbd}) follows from Lemma \ref{checkunitA}.
\end{proof}

\begin{prop}\label{boundarysol2}
Let $\mathbf{\sigma}=\{\sigma_{a}\}_{a\in I}$ be as in Table \ref{sigmaa}. Then $\mathbf{\sigma}$ is a gauge of the $Q$-system in a commutative ring $R$. 
\end{prop}

\begin{proof}
Note that $\sigma_a=e^{-2\pi i (\hat{\tau}_a |\rho)}$ for $\hat{\tau}_{a}\in \hat{P}, \,a\in I$ in Table \ref{taua}. For each $a\in I$, we have
$$
\prod_{b\in I} \sigma_b^{C_{a b}}=\prod_{b\in I} e^{-2\pi i (C_{a b}\hat{\tau}_b |\rho)}=e^{-2\pi i (\sum_{b\in I} C_{a b}\hat{\tau}_b |\rho)}=1,
$$
which follows from Lemma \ref{checkunitA}.
\end{proof}

\begin{table}
\caption{Gauge $\mathbf{\sigma}=\{\sigma_a\}_{a\in I}$}
\begin{center}
\label{sigmaa}
\begin{tabular}{c|c}
$\mathfrak{g}$ & $\sigma_a$   \\
\hline
$A_r$ & $-1$ if $a$ and $r$ are both odd
\\
\hline
$B_r$ & $-1$ if $a$ is odd
\\
\hline
$C_r$ & $-1$ if $r\equiv 1,2 \pmod 4$ and $a=r$
\\
\hline
$D_r$ & $-1$ if $r\equiv 2,3 \pmod 4$ and $a=r, r-1$
\\
\hline
$E_7$ & $-1$ if $a=4,6,7$
\\
\hline
otherwise & 1
\end{tabular}
\end{center}
\end{table}

\subsection{Conjecture on the image of KR modules in the fusion ring}
We begin with a simple observation about $\beta_k(\operatorname{res} W^{(a)}_{m})$ for small $m\geq 0$.

Recall (\ref{Qdecomp}) that 
$$
\operatorname{res} W^{(a)}_m=\sum_{\lambda \in P_{+}}Z(a, m,\lambda)V_{\lambda}.
$$ 
\begin{prop}\label{phalf}
If $0\le m \le \lfloor\frac{k}{c_a}\rfloor$, then $\beta_k(\operatorname{res} W^{(a)}_{m})$ is positive. If $m \geq \lfloor\frac{k}{c_a}\rfloor+1$, then there exists $\lambda\in P_{+}$ with $Z(a, m,\lambda)\neq 0$ such that its corresponding affine weight $\hat{\lambda}\in  \hat{P}^k$ is not dominant integral, i.e., $\hat{\lambda}\notin \hat{P}_{+}^k$. 
\end{prop}
\begin{proof}
This is essentially \cite[Corollary 3.3]{MR3282650}. Here we give a simplified proof. For $\lambda=\sum_{i=1}^{r}\lambda_{i}\omega_{i}\in P_{+}$, the corresponding affine weight $\hat{\lambda}$ is given by $\sum_{i=0}^{r}\lambda_{i}\hat{\omega}_{i}$ with $\lambda_0=k-(\lambda|\theta)$. As $(\theta|\alpha_i)\geq 0$ for any $i\in I$, we have 
$$mc_a=(m\omega_a|\theta)\geq (\lambda|\theta) = k-\lambda_0$$ for any $\lambda$ with $Z(a,m,\lambda)\neq 0$.

This inequality shows that if $0\le m \le \lfloor\frac{k}{c_a}\rfloor$, then $\lambda_0\geq 0$ and hence $\hat{\lambda}\in \hat{P}_{+}^k$ for all $\lambda\in P_{+}$ with $Z(a, m,\lambda)\neq 0$. When $m \geq \lfloor\frac{k}{c_a}\rfloor+1$, $\lambda = m \omega_a$ gives $\hat{\lambda} = \lambda_0 \hat{\omega}_0 + m \hat{\omega}_a$ where $\lambda_0 = k-mc_a<0$. This proves the proposition.
\end{proof}

This shows that when $m\geq \lfloor\frac{k}{c_a}\rfloor+1$, we have to implement (2.12) and carry out the necessary computation to get $\beta_k(\operatorname{res} W^{(a)}_{m})$. As we eventually want to know the positivity for $\beta_k(\operatorname{res}W^{(a)}_{m})$ for $0\le m \le t_a k$, we see that the number $c_at_a$ roughly measures the difficulty of the problem for each $a\in I$. Here's the table of $\max_{a\in I} c_at_a$ for each type :
$$
\begin{array}{c|c|c|c|c|c|c|c|c|c|c}
  & A_r & B_r & C_r & D_{r} & E_6 & E_7 & E_8 & F_4 & G_2 \\
\hline
\max_{a\in I} c_at_a & 1 & 2 & 2 & 2 & 3 & 4 & 6 & 4 & 3
\end{array}.
$$

Now we state the conjecture on $\beta_k(\operatorname{res}W^{(a)}_{m})$.
\begin{conjecture}\label{pconj}
For $a\in I$, let $\hat{\tau}_{a}\in \hat{P}$ be as in Table \ref{taua} and $\sigma_{a}=e^{-2\pi i (\hat{\tau}_a|\rho)}$. 
The following properties hold :
\begin{enumerate}[label=(\roman{*}), ref=(\roman{*})]

\item \label{pconj1} $\beta_k(\operatorname{res}W^{(a)}_{m})$ is positive for $0\le m \le t_a k$,
\item \label{pconj2} $\beta_k(\operatorname{res}W^{(a)}_{t_ak-m})=V_{k \hat{\tau}_a}\left(\beta_k(\operatorname{res}W^{(a)}_{m})\right)^{*}$ for $0\le m \le t_a k$,
\item \label{pconj4} $\beta_k(\operatorname{res}W^{(a)}_{t_ak+1})=\beta_k(\operatorname{res}W^{(a)}_{t_ak+2})=\cdots =\beta_k(\operatorname{res}W^{(a)}_{t_a(k+h^{\vee})-1})=0$,
\item \label{pconj5} $\beta_k(\operatorname{res}W^{(a)}_{m+nt_a(k+h^{\vee})})=\sigma_a^{n}V_{k \hat{\tau}_a}^n \beta_k(\operatorname{res}W^{(a)}_{m})$ for $0 \le m \le t_a(k+h^{\vee})-1$ and $n\in \mathbb{Z}_{\ge 0}$.
\end{enumerate}
\end{conjecture}
For the values of $\sigma_{a}=e^{-2\pi i (\hat{\tau}_a|\rho)}$, see Table \ref{sigmaa}. The conjecture is supported by many symbolic calculations involving the affine Weyl group. In Section \ref{perMS}, we will give a proof of Conjecture \ref{pconj} for some special cases; see Theorem \ref{simplepconj} for the statement. 
In Theorem \ref{catQ}, we show that $\{R^{(a)}_{m}\}_{(a,m)\in H_k}$ is a positive solution of the level $k$ restricted $Q$-system of $\mathfrak{g}$, which is of classical type, where
$$
R^{(a)}_{m}=\left\{
\begin{array}{ll}
\beta_k(\operatorname{res}W^{(a)}_{m}) & 0\le m \le \lfloor\frac{t_a k}{2}\rfloor\\
V_{k\hat{\tau}_{a}}\left(\beta_k(\operatorname{res} W^{(a)}_{t_a k-m})\right)^{*} & \lfloor\frac{t_a k}{2}\rfloor+1\le m\le t_a k
\end{array}.
\right.
$$ Then Propositions \ref{boundarysol} and \ref{boundarysol2} imply that $\{\sigma_a^{n}V_{k \hat{\tau}_a}^n R^{(a)}_{m}\}_{(a,m)\in H_k}$ is also a solution of the level $k$ restricted $Q$-system. These are all consistent with the above.

\begin{remark}\label{remK}
All these properties but the last one are a kind of synthesis and generalizations of various observations and conjectures in \cite{springerlink:10.1007/BF01840426, Kuniba1992,Kuniba1993,MR1304818,1751-8121-44-10-103001}. Our reformulation of them in terms of the fusion ring makes all these clear.

The positivity of $\mathcal{D}_{m}^{(a)}$ in \cite{springerlink:10.1007/BF01840426} can be naturally explained when $\beta_k(\operatorname{res}W^{(a)}_{m})$ is shown to be positive as in \ref{pconj1}. Kirillov's another claim, mentioned in the introduction of this paper, that $$\mathcal{D}_{t_a k+1}^{(a)}=Q^{(a)}_{t_a k+1}(\frac{\rho}{k+h^{\vee}})=0$$ is incorporated into \ref{pconj4}, 
where $Q^{(a)}_{t_a k+1}$ denotes the character of $\operatorname{res} W^{(a)}_{t_a k+1}$. The relation
$$
\mathcal{D}^{(a)}_{t_ak-m,\hat{\mu}}= \mathcal{D}^{(a)}_{t_ak,\hat{\mu}}\mathcal{D}^{(a)*}_{m,\hat{\mu}},
$$
which is a consequence of \ref{pconj2}, has been observed in \cite{Kuniba1992} from numerical tests. Our approach now gives a way to explain this. 

It is usually more difficult to prove certain properties of $\beta_k(\operatorname{res}W^{(a)}_{m})$ as an element of $\operatorname{Fus}_{k}(\mathfrak{g})$ than to establish the analogous property of $\mathcal{D}_{m}^{(a)}$ as a complex number. However even with some partial information about $\beta_k(\operatorname{res}W^{(a)}_{m})$ as in Theorem \ref{catQ}, we can prove some of the old conjectures about complex solutions as in Theorem \ref{qdimthm} and Corollary \ref{KNSconj}. 

The periodicity \ref{pconj5} has been pointed out in \cite{chlee2012} in a less precise form than given here for simply-laced types. It allows us to describe $\beta_k(\operatorname{res}W^{(a)}_{m})$ for all $m\ge 0$ completely from $\beta_k(\operatorname{res}W^{(a)}_{m})$ for $0\le m\le \lfloor\frac{t_a k}{2}\rfloor$. This is a revelation of linear recurrence relations \cite{dfk,2009arXiv0905.3776N} in the sequence $\{\chi(\operatorname{res} W^{(a)}_{m})\}_{m=0}^{\infty}$ where $\chi$ denotes the character, which is also conjectural in general; see \cite{2014arXiv1412.1638L} for the precise statement. We expect that the periodicity part will be settled once we have sufficient understanding of the coefficients in those relations. (And since submission of this paper, there has been some progress in this topic; see \cite{Lee2016}).

The complex solutions of level $k$ restricted $Q$-systems play important roles in Nahm's conjecture \cite{Nahm}, which attempts to give a partial answer to the question of when a certain form of $q$-hypergeometric series can be a modular function. In \cite{2009arXiv0905.3776N}, the authors try to find solutions $\{Q^{(a)}_{m}\}_{(a,m)\in H}$ of the unrestricted $Q$-systems of type $A_r$ and $D_r$ such that $Q^{(a)}_{k}=1$ and $Q^{(a)}_{k+1}=Q^{(a)}_{k+2}=\cdots =Q^{(a)}_{(k+h^{\vee})-1}=0$ to find all complex solutions of the level restricted $Q$-systems. As we can see in \ref{pconj2}, these are not quite the correct conditions one should impose as it can give only partial list among them. We also add that solutions of level restricted $T$-systems of type $A_r$ satisfying similar level truncation properties \ref{pconj4} and periodicity \ref{pconj5} have been studied in \cite{MR3015686}.

The positivity \ref{pconj1} now poses the problem of combinatorial description of coefficients of $V_{\hat{\lambda}}$ in the product of many copies of $\beta_k(\operatorname{res}W^{(a)}_{m})$ for $a\in H_k$ in the spirit of the Bethe ansatz or the fermionic formula \cite{MR1745263, MR1903978}. We can find this problem discussed only for $\operatorname{Fus}_{k}(\mathfrak{sl}_2)$ in \cite{Kirillov:1992ub}.

Clarifying the role of the fusion ring in the theory of cluster algebras will also be a problem of interest along the same lines of \cite{MR2682185, MR3500832}. Regarding \ref{pconj4}, one may ask if these elements generate the fusion ideal $\ker \beta_k$ \cite{MR2499554, MR3037582}. More broadly, understanding the role of $Q$-systems and the Kirillov-Reshetikhin modules in the study of fusion ideals will also be a topic of further research. 
\end{remark}

\begin{corollary}\label{percor}
Assume that Conjecture \ref{pconj} is true. For $(a,m)\in H$,
\begin{enumerate}[label=(\roman{*}), ref=(\roman{*})]

\item \label{signdef} $\beta_k(\operatorname{res}W^{(a)}_{m})$ is always either positive, negative or zero,
\item \label{percor2}$\beta_k(\operatorname{res}W^{(a)}_{m+Mt_{a}(k+h^{\vee})})=\beta_k(\operatorname{res}W^{(a)}_{m})$ where $M$ is given by the table 
$$
\begin{array}{c|c|c|c|c|c|c|c|c|c|c}
  & A_r & B_r & C_r & D_{2l} & D_{2l+1} & E_6 & E_7 & E_8 & F_4 & G_2 \\
\hline
 M & r+1 & 2 & 2 & 2 & 4 & 3 & 2 & 1 & 1 & 1
\end{array}.
$$
\end{enumerate}
\end{corollary}
\begin{proof}
\ref{signdef} is a consequence of Conjecture \ref{pconj} \ref{pconj1}, \ref{pconj4} and \ref{pconj5}.
To prove \ref{percor2}, we can use Conjecture \ref{pconj} \ref{pconj5}, together with the fact $\sigma_a^{M}=V_{k\hat{\tau}_{a}}^M=1$.
\end{proof}

\subsection{Proof of Conjecture \ref{pconj} in some cases}\label{perMS}
The following property of the modular $S$-matrix has been noticed in \cite{Spiegelglas1990}. We make the statement in a precise form.

\begin{prop}\label{addsigma}
Let $\lambda$, $\mu$ and $\sigma$ be elements of $P$. If $\lambda'=\lambda+(k+h^{\vee})\sigma$, then
$$
S_{\hat{\lambda}', \hat{\mu}}=e^{-2\pi i (\sigma|\mu+\rho)}S_{\hat{\lambda},\hat{\mu}}.
$$
In particular, when $\sigma\in Q^{\vee}\subseteq P$ is a coroot,
$$
S_{\hat{\lambda}',\hat{\mu}}=S_{\hat{\lambda},\hat{\mu}}.
$$
\end{prop}

\begin{proof}
For $w\in W$, we have
$$
(w \sigma|\mu+\rho)=(\sigma|\mu+\rho) \pmod{\mathbb{Z}}.
$$
Then from (\ref{modularS}), we can pull out the factor $e^{-2\pi i (\sigma|\mu+\rho)}$ so that
$$S_{\hat{\lambda}', \hat{\mu}}=e^{-2\pi i (\sigma|\mu+\rho)}S_{\hat{\lambda},\hat{\mu}}.$$
If $\sigma\in Q^{\vee}$, then the fact that $(\sigma|\mu+\rho)\in \mathbb{Z}$ implies $e^{-2\pi i (\sigma|\mu+\rho)}=1$.
\end{proof}

\begin{lemma}\label{prdlma}
Let $\tau\hat{\omega}_0=\hat{\omega}_{a}$ for some $\tau \in O(\hat{\mathfrak{g}})$ and $a\in I$. Let $\zeta= e^{-2\pi i (\omega_{a}|\rho)}$. If $\lambda=m\omega_{a}$ and $\lambda'=\lambda+n (k+h^{\vee})\omega_{a}$, then 
$$
S_{\hat{\lambda}', \hat{\mu}}= \zeta^n S_{\tau^n \hat{\lambda}, \hat{\mu}}
$$
for any $\hat{\mu}\in \hat{P}_{+}^{k}$.
\end{lemma}

\begin{proof}
If we apply Proposition \ref{addsigma} for $\sigma=n\omega_{a}$, then we get
$$
\begin{aligned}
S_{\hat{\lambda}', \hat{\mu}} &= e^{-2\pi i(n \omega_{a}|\mu+\rho)}S_{\hat{\lambda}, \hat{\mu}} \\
{} &= e^{-2\pi in (\omega_{a}|\rho)}e^{-2\pi in (\omega_{a}|\mu)}S_{\hat{\lambda} ,\hat{\mu}} \\
{} &= \zeta^n S_{\tau^n\hat{\lambda}, \hat{\mu}}.
\end{aligned}
$$
We have used (\ref{outAutS}) to get the equality in the last line. This proves our lemma.
\end{proof}

This result can be applied for any $a\in I$ listed in Table \ref{svertices}.
\begin{table}
\caption{Vertices $a\in I$ such that $\operatorname{res} W^{(a)}_m=V_{m \omega_a}$}
\label{svertices}
$$
\begin{array}{c|c|c|c|c|c|c|c|c|c|c}
 \mathfrak{g} & A_r & B_r & C_r & D_{r} & E_6 & E_7  \\
\hline
 a & 1,\cdots,r & 1 & r & 1,r,r-1 & 1,5 & 6
\end{array}
$$
\end{table}

\begin{prop}\label{Croot}
Let $a\in I$ be as in Table \ref{svertices}. For each integer $l$ such that $1 \le l \le h^{\vee}-1$, there exists a positive root $\alpha$ such that  $(\omega_a|\alpha)=1$ and $(\rho|\alpha)=l$.
\end{prop}

Our proof is based on a case-by-case check for each root system. Since it is straightforward and lengthy, we give it in Appendix \ref{app}.

\begin{theorem}\label{zeroS}
Let $a\in I$ be as in Table \ref{svertices}. Then $\beta_k(V_{(k-m)\hat{\omega}_0+m \hat{\omega}_{a}})=0$ for $k+1 \le m \le k+h^{\vee}-1$.
\end{theorem}

\begin{proof}
By Proposition \ref{Croot}  and the product formula (\ref{quanProd}) for the quantum dimension, we get $\mathcal{D}_{(k-m)\hat{\omega}_0+m \hat{\omega}_{a}}=0$. Then by Proposition \ref{pro:quantumD}, 
$$
\operatorname{qdim}_{\hat{\mu}}V=\frac{S_{(k-m)\hat{\omega}_0+m \hat{\omega}_{a},\hat{\mu}}}{S_{\hat{0},\hat{\mu}}}=0
$$ 
for all $\hat{\mu}\in \hat{P}_{+}^{k}$. Then our statement follows from Proposition \ref{prop:qdimmu} \ref{itm4:prop:qdimmu}.
\end{proof}

\begin{theorem}\label{simplepconj}
Let $a\in I$ as in Table \ref{svertices}. Conjecture \ref{pconj} holds true. In particular, it holds for all $a\in I$ in the case of type $A_r$.
\end{theorem}
\begin{proof}
First note that for $a\in I$, we have $t_a=1$ and $c_a=1$. The positivity \ref{pconj1} follows from Proposition \ref{phalf}. It is straightforward to check
\begin{equation}\label{stck}
\begin{aligned}
V_{k\hat{\tau}_a} \left(\beta_k(\operatorname{res} W^{(a)}_m)\right)^{*} & = V_{k\hat{\tau}_a} V^{*}_{(k-m)\hat{\omega}_0+m \hat{\omega}_a} \\
{} & =V_{m \hat{\omega}_0+(k-m)\hat{\omega}_a}\\
{} & =\beta_k(\operatorname{res} W_{k-m}^{(a)})
\end{aligned}
\end{equation}
for all $0\le m \le k$. See Table \ref{tauprimeB} for the equality in the second line of (\ref{stck}). This proves the symmetry condition \ref{pconj2}.
\ref{pconj4} follows from Theorem \ref{zeroS}.

By Lemma \ref{prdlma}, for $0 \le m \le k+h^{\vee}-1$ and $n\in \mathbb{Z}_{\ge 0}$, we have
$$
\operatorname{qdim}_{\hat{\mu}}\beta_k(\operatorname{res} W^{(a)}_{m+n (k+h^{\vee})})=\operatorname{qdim}_{\hat{\mu}}\sigma_a^{n}V_{k\hat{\tau}_{a}}^n \beta_k(\operatorname{res}W^{(a)}_{m})
$$
for all $\hat{\mu}\in \hat{P}_{+}^{k}$. 
Then by Proposition \ref{prop:qdimmu} \ref{itm4:prop:qdimmu}, 
$$\beta_k(\operatorname{res} W^{(a)}_{m+n (k+h^{\vee})})=\sigma_a^{n}V_{k\hat{\tau}_{a}}^n \beta_k(\operatorname{res} W^{(a)}_{m}).$$
This proves the periodicity condition \ref{pconj5}.
\end{proof}
\begin{table}
\caption{Diagram automorphisms corresponding to $V\mapsto V_{k \hat{\tau}_a}V^{*}$} 
\begin{center}
\label{tauprimeB}
\begin{tabular}{c|c}
$\mathfrak{g}$ & diagram automorphisms \\
\hline
$A_r$ & 
$
\left(
\begin{array}{cccccccc}
 0 & 1 & \cdots & {a-1} & a & {a+1} & \cdots & r  \\
 a & {a-1} & \cdots & 1 & 0 & r & \cdots & {a+1}\\
\end{array}
\right)
$
\\
\hline
$B_r$ & $(0 \quad 1) \quad \text{if $a\equiv 1 \pmod 2$}$
\\
\hline
$C_r$ & 
$\left(
\begin{array}{ccccccc}
 0 & 1 & 2 & \cdots & {r-2} & {r-1} & r \\
 {r} & {r-1} & {r-2} & \cdots & 2 & 1 & 0 \\
\end{array}
\right) \quad \text{if $a=r$}$
\\
\hline
$D_{r}$, $r$ even &
$
\begin{cases}
(0\quad 1)({r-1}\quad r) \quad \text{if } 1\le a \le r-2 \text{ and } a\equiv 1 \pmod 2
\\
\left(
\begin{array}{ccccccc}
 0 & 1 & 2 & \cdots & {r-2} & {r-1} & r \\
 {r-1} & r & {r-2} & \cdots & 2 & 0 & 1 \\
\end{array}
\right) & \text{if $a=r-1$}
\\
\left(
\begin{array}{ccccccc}
 0 & 1 & 2 & \cdots & {r-2} & {r-1} & r \\
 {r} & {r-1} & {r-2} & \cdots & 2 & 1 & 0 \\
\end{array}
\right) & \text{if $a=r$}
\end{cases}
$
\\
\hline
$D_{r}$, $r$ odd &
$
\begin{cases} 
(0\quad 1) \quad \text{if } 1\le a \le r-2 \text{ and } a\equiv 1 \pmod 2
\\
\left(
\begin{array}{ccccccc}
 0 & 1 & 2 & \cdots & {r-2} & {r-1} & r \\
 {r-1} & r & {r-2} & \cdots & 2 & 0 & 1 \\
\end{array}
\right) & \text{if $a=r-1$}
\\
\left(
\begin{array}{ccccccc}
 0 & 1 & 2 & \cdots & {r-2} & {r-1} & r \\
 {r} & {r-1} & {r-2} & \cdots & 2 & 1 & 0 \\
\end{array}
\right) & \text{if $a=r$}
\end{cases}
$
\\
\hline
$E_{6}$ &
$
\begin{cases} 
(0 \quad 1)(2 \quad 6)
& \text{if $a=1,4$}\\ 
(0 \quad 5)(4 \quad 6)
& \text{if $a=2,5$}\\
(1 \quad 5)(2 \quad 4)
& \text{if $a=3,6$}
\end{cases}
$
\\
\hline
$E_{7}$ &
$
(0 \quad 6)(1 \quad 5)(2 \quad 4)
\quad \text{if $a=4,6,7$}
$
\\
\hline
otherwise & trivial
\end{tabular}
\end{center}
\end{table}
\section{Level $k$ restricted $Q$-systems in $\operatorname{Fus}_{k}(\mathfrak{g})$}\label{sec:possol}
The main result of this section is Theorem \ref{catQ} where we construct a positive solution of the level $k$ restricted $Q$-system in $\operatorname{Fus}_{k}(\mathfrak{g})$ when $\mathfrak{g}$ is of classical type. First we prove an important stepping stone to prove Theorem \ref{catQ}.
\begin{lemma}\label{centrosym}
Let $\mathfrak{g}$ be of types $A_r,B_r,C_r$ and $D_r$. Let $s=\lfloor\frac{t_a k}{2}\rfloor$. 
If $t_a k$ is even, then $\beta_k(\operatorname{res} W_{s+1}^{(a)})=V_{k\hat{\tau}_{a}}\left(\beta_k(\operatorname{res} W^{(a)}_{s-1})\right)^{*}$. If $t_a k$ is odd, then $\beta_k(\operatorname{res} W_{s+1}^{(a)})=V_{k\hat{\tau}_{a}}\beta_k(\operatorname{res}W_{s}^{(a)})$.
\end{lemma}
\begin{proof}
For type $A_r$, the statement follows from Theorem $\ref{simplepconj}$. For type $D_r$, Theorem $\ref{simplepconj}$ and the arguments in \cite[Propositions 3.3, 3.4, 3.6, 3.8]{Lee2012} can be used to prove the lemma. For the remaining cases, the method of the proof is exactly the same as in the case of type $D_r$. Since it is mainly a laborious case-by-case check, we will give a proof for types $B_r$ and $C_r$ in Appendix \ref{proofB} and \ref{proofC}, respectively.
\end{proof}
Now we can construct a positive solution of the level $k$ restricted $Q$-system in $\operatorname{Fus}_{k}(\mathfrak{g})$. The basic idea is to glue two different solutions of the unrestricted $Q$-system from the opposite directions to form a single solution of the level $k$ restricted $Q$-system. In order to glue them consistently at the intersection, we need to employ Lemma \ref{centrosym}. This idea has been used in \cite{Lee2012} to obtain a positive real solution of the level restricted $Q$-system of type $D_r$. The table after Proposition \ref{phalf} indicates why this method does not apply to other exceptional types.
\begin{definition}\label{def:Rma}
For the brevity of notation, let us write $\beta_k(\operatorname{res}W^{(a)}_{m})$ as $Z^{(a)}_{m}$ for the rest of this section. For each $a\in I$, let us define $R^{(a)}_{m}\in \operatorname{Fus}_{k}(\mathfrak{g})$ by
$$
R^{(a)}_{m}=\left\{
\begin{array}{ll}
Z^{(a)}_{m} & 0\le m \le \lfloor\frac{t_a k}{2}\rfloor\\
V_{k\hat{\tau}_{a}}Z^{(a)*}_{t_a k-m} & \lfloor\frac{t_a k}{2}\rfloor+1\le m\le t_a k
\end{array}
\right.
$$ and $R^{(a)}_{-1}=R^{(a)}_{t_a k+1}=0$.
\end{definition}

\begin{theorem}\label{catQ}
Let $\mathfrak{g}$ of types $A_r,B_r,C_r$ and $D_r$. Then $\{R^{(a)}_{m}\}_{(a,m)\in H_k}$ is a positive solution of the level $k$ restricted $Q$-system of type $\mathfrak{g}$.
\end{theorem}

\begin{proof}
The positivity of $R^{(a)}_{m},\, (a,m)\in H_k$ is clear from Proposition \ref{phalf}. We have to check that the equality
\begin{equation}
\left(R^{(a)}_{m}\right)^2 = R^{(a)}_{m-1}R^{(a)}_{m+1} + 
\prod_{b:b \sim a} \prod_{j=0}^{-C_{a b}-1}
R^{(b)}_{\lfloor\frac{C_{b a}m - j}{C_{a b}}\rfloor} \label{checkQ}
\end{equation}
holds true for each $0 \le m \le t_a k$.
For $0 \le m \le \lfloor\frac{t_a k}{2}\rfloor-1$, it follows from Proposition \ref{pr:betaim}. For $m=\lfloor\frac{t_a k}{2}\rfloor$, we need to use Lemma \ref{centrosym} together with Proposition \ref{pr:betaim}.

Note that $\{Z^{(a)*}_{m}\}_{(a,m)\in H}$ is also a solution of the unrestricted $Q$-system. By Proposition \ref{gauge}, $\{V_{k\hat{\tau}_{a}}Z_{m}^{(a)*}\}_{(a,m)\in H}$ is also a solution of the unrestricted $Q$-system. In concrete terms, we have
\begin{equation}\label{checkW}
\left(V_{k\hat{\tau}_{a}}Z^{(a)*}_{m}\right)^2 =\left(V_{k\hat{\tau}_{a}}Z^{(a)*}_{m-1}\right)\left(V_{k\hat{\tau}_{a}}Z^{(a)*}_{m+1}\right)+ 
\prod_{b:b \sim a} \prod_{j=0}^{-C_{a b}-1}
\left(V_{k\hat{\tau}_{b}}Z^{(b)*}_{\lfloor\frac{C_{b a}m - j}{C_{a b}}\rfloor}\right)
\end{equation}
for each $m\geq 0$. Since this holds especially for $0\le m \le \lfloor\frac{t_a k}{2}\rfloor-1$, (\ref{checkQ}) is true for $\lfloor\frac{t_a k+1}{2}\rfloor+1 \le m \le t_a k$. The only possible value not verified so far, which happens when $t_ak$ is odd, is $m=\lfloor\frac{t_a k+1}{2}\rfloor$. For $m=\lfloor\frac{t_a k+1}{2}\rfloor$, again we can use Lemma \ref{centrosym} together with (\ref{checkW}) for $V_{k\hat{\tau}_{a}}Z^{(a)*}_{m}$.
\end{proof}

\begin{remark}
Conjecture \ref{pconj} claims that $\{\beta_k(\operatorname{res}W^{(a)}_{m})\}_{(a,m)\in H_k}$ is a positive solution of the level $k$ restricted $Q$-system. We believe that this is a unique positive solution up to certain obvious symmetries like the conjugation $*$ and a gauge. Lemma \ref{centrosym} proves that $R^{(a)}_{m}=\beta_k(\operatorname{res}W^{(a)}_{m})$ for $0\le m\le \lfloor\frac{t_a k}{2}\rfloor+1$ in all classical types. However, the identity $R^{(a)}_{m}=\beta_k(\operatorname{res} W^{(a)}_{m})$ for $\lfloor\frac{t_a k}{2}\rfloor+2\le m\le t_a k$ still remains to be proved in general except some cases, for example, the cases covered in Theorem \ref{simplepconj}.
\end{remark}

\section{Level $k$ restricted $Q$-systems in $\mathbb{C}$}\label{sec:appsfusion}
In this section, we study complex solutions of the level $k$ restricted $Q$-systems. 
\subsection{String of zeros}
In this subsection, we investigate some conditions under which we have a string of zeros like \ref{pconj4} of Conjecture \ref{pconj} in a complex solution of $Q$-systems. For the rest of this section, we assume that $\{Q^{(a)}_{m}\}_{(a,m)\in H}$ is a complex solution of the $Q$-system of type $\mathfrak{g}$. 

\begin{lemma}\label{string0a}
Let $m\in \mathbb{Z}_{\ge 0}$. Suppose that $Q^{(a)}_{t_a m+1}=0$ for all $a\in I$. Then $Q^{(a)}_{t_a m+1}=\cdots = Q^{(a)}_{t_a (m+1)}=0$ for all $a\in I$.
\end{lemma}
\begin{proof}
This is a consequence of (\ref{eq:qsys}).
\end{proof}

\begin{lemma}\label{string0}
Let $m\in \mathbb{Z}_{\ge 0}$. Suppose $Q^{(a)}_{t_a m}=0$ for all $a\in I$. Then $Q^{(a)}_{t_a m+1}=\cdots =Q^{(a)}_{t_a (m+1)-1}=0$ for all $a\in I$.
\end{lemma}
\begin{proof}
It also follows from (\ref{eq:qsys}) easily.
\end{proof}

\begin{lemma}\label{string0b}
Let $m\ge 1$. Assume that $Q^{(a)}_{t_a m-1}\neq 0$ for all $a\in I$ and $\{Q^{(a)}_{t_a m}\}_{a\in I}$ is a gauge. Then $Q^{(a)}_{t_a m+1}=\cdots =Q^{(a)}_{t_a (m+2)-1}=0$ for all $a\in I$.
\end{lemma}
\begin{proof}
The equation (\ref{eq:qsys})
$$
\left(Q^{(a)}_{t_a m}\right)^2 = Q^{(a)}_{t_a m-1}Q^{(a)}_{t_a m+1} + \prod_{b:b \sim a} \left(Q^{(b)}_{t_b m}\right)^{-C_{a b}}
$$ implies $Q^{(a)}_{t_a m+1}=0$. Thus Lemma $\ref{string0a}$ implies
$$
Q^{(a)}_{t_a m+1}=\cdots = Q^{(a)}_{t_a (m+1)}=0
$$ for all $a\in I$. 
Then we obtain
$$
Q^{(a)}_{t_a (m+1)+1}=\cdots =Q^{(a)}_{t_a (m+2)-1}=0
$$
from Lemma \ref{string0}.
\end{proof}

\begin{lemma}\label{string0c}
Let $m\in \mathbb{Z}_{\ge 0}$. Suppose $Q^{(a)}_{t_a m}=\cdots =Q^{(a)}_{t_a (m+1)-1}=0$ for all $a\in I$ and $Q^{(b)}_{t_b(m+1)}=0$ for at least one vertex $b\in I$. Then $Q^{(a)}_{t_a (m+1)}=\cdots=Q^{(a)}_{t_a(m+2)-1}=0$  for all $a\in I$.
\end{lemma}
\begin{proof}
We can show that $Q^{(a)}_{t_a(m+1)}=0$ for all $a\in I$ using (\ref{eq:qsys}). Then the lemma follows from Lemma $\ref{string0}$.
\end{proof}
\subsection{Complex solutions of level $k$ restricted $Q$-systems}
Let us begin with a simple lemma.
\begin{lemma}\label{zeq}
Let $\mathbf{w}=\{w^{(a)}_{m}\}_{(a,m)\in H_k}$ be a complex solution of the level $k$ restricted $Q$-system such that $w^{(a)}_{m}\neq 0$ for any $(a,m)\in H_k$. If $\mathbf{z}=\{z^{(a)}_{m}\}_{(a,m)\in H}$ is a solution of the unrestricted $Q$-system and $w^{(a)}_{1}=z^{(a)}_{1}$ for any $a\in I$, then $w^{(a)}_{m}=z^{(a)}_{m}$ for $0\le m \le t_a k$.
\end{lemma}

\begin{proof}
This is a direct consequence of the recursion (\ref{eq:qsys}).
\end{proof}
\begin{theorem}\label{qdimthm}
Let $\mathfrak{g}$ be of types $A_r,B_r,C_r$ and $D_r$ and let $\hat{\mu}\in \hat{P}_{+}^{k}$. Assume that $\mathcal{D}^{(a)}_{m,\hat{\mu}}\neq 0$ for all $a\in I$ and $0\le m \le \lfloor \frac{t_a k}{2}\rfloor$. 
Then the following properties hold for each $a\in I$ : 
\begin{enumerate}[label=(\roman{*}), ref=(\roman{*})]

\item \label{qdimthm:WR} $\mathcal{D}^{(a)}_{m,\hat{\mu}}=\operatorname{qdim}_{\hat{\mu}} R^{(a)}_{m}$ for $0\le m \le t_a k$,
\item \label{qdimthm:sym} $\mathcal{D}^{(a)}_{m,\hat{\mu}}= e^{-2\pi i (\hat{\tau}_a |\mu)}\mathcal{D}^{(a)*}_{t_ak-m,\hat{\mu}}$ for $0\le m \le t_a k$,
\item \label{qdimthm:zero} $\mathcal{D}^{(a)}_{t_a k+1,\hat{\mu}}=\mathcal{D}^{(a)}_{t_a k+2,\hat{\mu}}=\cdots =\mathcal{D}^{(a)}_{t_a(k+h^{\vee})-1,\hat{\mu}}=0$.
\end{enumerate}
\end{theorem}

\begin{proof}
From the assumption that $\operatorname{qdim}_{\hat{\mu}}(R^{(a)}_{m})=\mathcal{D}^{(a)}_{m,\hat{\mu}}\neq 0$ for all $0\le m \le \lfloor \frac{t_a k}{2}\rfloor$, we have $\operatorname{qdim}_{\hat{\mu}}(R^{(a)}_{m})\neq 0$ for all $0\le m \le t_a k$ by Theorem \ref{catQ} and Proposition \ref{prop:qdimmu}. Thus we can conclude that $\operatorname{qdim}_{\hat{\mu}}(R^{(a)}_{m})$ must be equal to $\mathcal{D}^{(a)}_{m,\hat{\mu}}$ for $0\le m \le t_a k$ by Lemma \ref{zeq}. This proves \ref{qdimthm:WR} and \ref{qdimthm:sym}.

Since $\mathcal{D}^{(a)}_{t_a k-1,\hat{\mu}}\neq 0$ for each $a\in I$ by \ref{qdimthm:sym} and $\{\mathcal{D}^{(a)}_{t_a k,\hat{\mu}}\}_{a\in I}$ is a gauge, we get
$$
\mathcal{D}^{(a)}_{t_a k+1,\hat{\mu}}=\cdots =\mathcal{D}^{(a)}_{t_a (k+1),\hat{\mu}}=0
$$
for any $a\in I$ by Lemma \ref{string0b}. Since the Dynkin diagram of $\mathfrak{g}$ has at least one vertex listed in Table \ref{svertices}, there exists $b\in I$ such that
$$\mathcal{D}^{(b)}_{t_b k+1,\hat{\mu}}=\mathcal{D}^{(b)}_{t_b k+2,\hat{\mu}}=\cdots =\mathcal{D}^{(b)}_{t_b(k+h^{\vee})-1,\hat{\mu}}=0$$
by Theorem \ref{simplepconj} and Proposition \ref{prop:qdimmu} \ref{itm4:prop:qdimmu}. Thus the assumptions of Lemma \ref{string0c} are now all satisfied and we can conclude that
$$
\mathcal{D}^{(a)}_{t_a k+1,\hat{\mu}}=\mathcal{D}^{(a)}_{t_a k+2,\hat{\mu}}=\cdots =\mathcal{D}^{(a)}_{t_a(k+h^{\vee})-1,\hat{\mu}}=0
$$
for any $a\in I$. We thus have proved \ref{qdimthm:zero}.
\end{proof}

\begin{remark}
If we can establish the above result without assuming $\mathcal{D}^{(a)}_{m,\hat{\mu}}\neq 0$ for $a\in I$ and $0\le m \le \lfloor \frac{t_a k}{2}\rfloor$, then we can prove the corresponding statements for the solution in  $\operatorname{Fus}_{k}(\mathfrak{g})$.
\end{remark}

Now we have a proof of the conjecture of Kirillov \cite{springerlink:10.1007/BF01840426} and Kuniba, Nakanishi and Suzuki \cite[Conjecture 14.2]{1751-8121-44-10-103001} for all classical types. 
\begin{corollary} \label{KNSconj} 
Let $\mathfrak{g}$ be of types $A_r,B_r,C_r$ and $D_r$. For each $a\in I$, the following properties hold :
\begin{enumerate}[label=(\roman{*}), ref=(\roman{*})]

\item \label{KNSconj:1} $\mathcal{D}_{m}^{(a)}>0$ for $0\le m \le t_a k$,
\item $\mathcal{D}_{m}^{(a)}=\mathcal{D}_{t_a k-m}^{(a)}$ for $0\le m \le t_a k$,
\item \label{KNSconj:5} $\mathcal{D}^{(a)}_{t_a k+1}=\mathcal{D}^{(a)}_{t_a k+2}=\cdots =\mathcal{D}^{(a)}_{t_a(k+h^{\vee})-1}=0$,
\item \label{KNSconj:ineq} $\mathcal{D}_{m-1}^{(a)}<\mathcal{D}_{m}^{(a)}$ for $1\le m\le \lfloor \frac{t_a k}{2}\rfloor$.
\end{enumerate}
\end{corollary}

\begin{proof}
We know that $\mathcal{D}^{(a)}_{m}=\operatorname{qdim} R^{(a)}_{m}> 0$ for all $0 \le m \le \lfloor \frac{t_a k}{2}\rfloor$. Then by Theorem \ref{qdimthm} \ref{qdimthm:WR}, we have $\mathcal{D}^{(a)}_{m}=\operatorname{qdim} R^{(a)}_{m}> 0$ for all $0 \le m \le t_a k$.
Now all the properties \ref{KNSconj:1}-\ref{KNSconj:5} follow from Theorem \ref{qdimthm} as a special case when $\hat{\mu}=k\hat{\omega}_0$.

Now we prove the inequality part \ref{KNSconj:ineq}. 
Let us define $\{x^{(a)}_m\}_{(a,m)\in \mathring{H}_k}$ as
\begin{equation}\label{eq:xadef}
x_{m}^{(a)}=\frac{\mathcal{D}_{m-1}^{(a)}\mathcal{D}_{m+1}^{(a)}}{(\mathcal{D}_{m}^{(a)})^2},\quad (a,m) \in \mathring{H}_k.
\end{equation}
As $\{\mathcal{D}^{(a)}_{m}\}_{(a,m)\in H_k}$ is a positive real solution of the level $k$ restricted $Q$-system, we have
\begin{equation}\label{eq:xa}
0<x_{m}^{(a)}<1,\quad (a,m) \in \mathring{H}_k.
\end{equation}
Let $s=\lfloor \frac{t_a k}{2}\rfloor$ for the rest of the proof. If $t_a k$ is odd, by the symmetry of solutions, we have $\mathcal{D}_{s+1}^{(a)}=\mathcal{D}_{s}^{(a)}$. Then
from $x_{s}^{(a)}<1$, we have
$$x_{s}^{(a)}=\frac{\mathcal{D}_{s-1}^{(a)}\mathcal{D}_{s+1}^{(a)}}{\mathcal{D}_{s}^{(a)}\mathcal{D}_{s}^{(a)}}=\frac{\mathcal{D}_{s-1}^{(a)}}{\mathcal{D}_{s}^{(a)}}<1.$$ If $t_a k$ is even, again the symmetry condition implies $\mathcal{D}_{s+1}^{(a)}=\mathcal{D}_{s-1}^{(a)}$. Using this, we get $$x_{s}^{(a)}=\frac{\mathcal{D}_{s-1}^{(a)}\mathcal{D}_{s+1}^{(a)}}{\mathcal{D}_{s}^{(a)}\mathcal{D}_{s}^{(a)}}=\frac{(\mathcal{D}_{s-1}^{(a)})^2}{(\mathcal{D}_{s}^{(a)})^2}=\left(\frac{\mathcal{D}_{s-1}^{(a)}}{\mathcal{D}_{s}^{(a)}}\right)^2<1.$$

In both cases, we obtain the inequality $\frac{\mathcal{D}^{(a)}_{s-1}}{\mathcal{D}^{(a)}_s}<1$. Then (\ref{eq:xa}) implies
$$\frac{\mathcal{D}^{(a)}_0}{\mathcal{D}^{(a)}_1}<\frac{\mathcal{D}^{(a)}_1}{\mathcal{D}^{(a)}_2}<\cdots<\frac{\mathcal{D}^{(a)}_{s-1}}{\mathcal{D}^{(a)}_s}<1.$$
This proves the inequality
$$
\mathcal{D}^{(a)}_0<\mathcal{D}^{(a)}_1<\cdots<\mathcal{D}^{(a)}_s.
$$
\end{proof}

\begin{remark}
The corresponding statement is true in type $E_6$ and partially in type $E_7$ and $E_8$; see \cite{MR3282650}. The quantities $\{f_{m}^{(a)}\}_{(a,m) \in \mathring{H}_k}$ where $f_{m}^{(a)}=1-x_{m}^{(a)}$ in (\ref{eq:xadef}) are the arguments for the dilogarithm identities in \cite{springerlink:10.1007/BF01840426}, as we mentioned earlier. These numbers also define certain torsion elements of the Bloch group of an appropriate number field \cite{LeeCNTP2013}.
\end{remark}

\begin{corollary}
For each $(a,m)\in H_k$, let $\mathcal{A}_{m}^{(a)}$ be the fusion matrix of $R_{m}^{(a)}$. Then
\begin{enumerate}[label=(\roman{*}), ref=(\roman{*})]

\item $\mathcal{A}_{m}^{(a)}$ is a non-negative integral matrix,
\item the family $\{\mathcal{A}_{m}^{(a)}\}_{(a,m)\in H_k}$ is a solution of the level $k$ restricted $Q$-system in a certain commutative subring of the ring of the square matrices of size $|\hat{P}_{+}^{k}|$ over $\mathbb{Z}$,
\item under the same assumptions as in Theorem \ref{qdimthm},  $\mathcal{D}^{(a)}_{m,\hat{\mu}}$ is an eigenvalue of $\mathcal{A}_{m}^{(a)}$ and in particular, $\mathcal{D}_{m}^{(a)}$ is the Perron-Frobenius eigenvalue of it.
\end{enumerate}
\end{corollary}
\begin{proof}
It follows from Theorems \ref{catQ} and \ref{qdimthm} and Proposition \ref{Perron-Frobenius}.
\end{proof}

\begin{remark}
In \cite[Section 3.7]{1751-8121-44-10-103001}, $\mathcal{A}_{m}^{(a)}$ is called the admissibility matrix of $\operatorname{res} W^{(a)}_{m}$.
\end{remark}

\appendix
\newtheorem{aprop}[theorem]{Proposition}[section]
\section{Proof of Proposition \ref{Croot}}\label{app}
For each $a\in I$ in Table \ref{svertices} and $l=1,\cdots, h^{\vee}-1$, we will construct a positive root $\beta_l$ such that $(\omega_a|\beta_l)=1$ and $(\rho|\beta_l)=l$.
\subsection*{Type $A_r$}
We can use induction on $r$. The statement is true for $A_1$. Suppose that $r>1$. 
Let $\beta_{h^{\vee}-1}=\theta$ be the highest root $\alpha_1+\cdots+\alpha_r$. We have $(\omega_a|\beta_{h^{\vee}-1})=1$ and $(\rho|\beta_{h^{\vee}-1})=h^{\vee}-1$. We can choose a simple root $\alpha_j \neq \alpha_a$ such that $(\beta_{h^{\vee}-1}|\alpha_j)=1$. Note that the only possible choices are $\alpha_j=\alpha_1$ or $\alpha_j=\alpha_r$. Then $\Pi- \{\alpha_j\}$ forms a simple system of type $A_{r-1}$ with the highest root $\beta_{h^{\vee}-2}=\beta_{h^{\vee}-1}-\alpha_j$. Using induction hypothesis, we can construct a sequence of roots $\beta_{h^{\vee}-2},\cdots, \beta_{1}=\alpha_a$ satisfying the conditions $(\omega_a|\beta_l)=1$ and $(\rho|\beta_l)=l$. Thus we have constructed a sequence of roots with desired properties. 

\subsection*{Type $B_r$ ($r\ge 2$)}
We have only one vertex $a=1$ in Table \ref{svertices}. For $1\le l \le r$, we define $\beta_{h^{\vee}-l}$ as follows :  
\begin{itemize}

  \item $\beta_{h^{\vee}-1}=\theta=\alpha_1+2\sum_{j=2}^{r}\alpha_j$ 
  \item $\beta_{h^{\vee}-2}=s_2  \beta_{h^{\vee}-1}=\alpha_1+\alpha_2+2\sum_{j=3}^{r}\alpha_j$
  \item $\beta_{h^{\vee}-3}=s_3  \beta_{h^{\vee}-2}=\alpha_1+\alpha_2+\alpha_3+2\sum_{j=4}^{r}\alpha_j$
  \item $\cdots$
  \item $\beta_{h^{\vee}-(r-1)}=s_{r-1}\beta_{h^{\vee}-(r-2)}= \sum_{j=1}^{r-1}\alpha_j+2\alpha_r$
  \item $\beta_{h^{\vee}-r}=s_r \beta_{h^{\vee}-(r-1)}=\sum_{j=1}^{r-1}\alpha_j$
\end{itemize}

One can see that the conditions $(\omega_1|\beta_l)=1$ and $(\rho|\beta_l)=l$ are satisfied for each $1\le l \le r$.
Note that $\Pi- \{\alpha_r\}$ forms a simple system of type $A_{r-1}$ with the highest root $\beta_{h^{\vee}-r}=\beta_{r-1}$
, we can now imitate the construction for type $A_{r-1}$ to define the rest of the terms $\beta_{r-2},\cdots, \beta_{1}=\alpha_1$. 

\subsection*{Type $C_r$ ($r\ge 2$)}
Let $a=r$. For $1\le l \le r$, we define $\beta_{h^{\vee}-l}$ as follows :  
\begin{itemize}

\item $\beta_{h^{\vee}-1}=\theta=\alpha_r+2\sum_{j=1}^{r-1}\alpha_j$
\item $\beta_{h^{\vee}-2}=s_1 \beta_{h^{\vee}-1}= \alpha_r+2\sum_{j=2}^{r-1}\alpha_j$
\item $\cdots$
\item $\beta_{h^{\vee}-(r-1)}=s_{r-2} \beta_{h^{\vee}-(r-2)}=\alpha_r+2\alpha_{r-1}$
\item $\beta_{h^{\vee}-r}=s_{r-1}\beta_{h^{\vee}-(r-1)}=\alpha_{r}$
\end{itemize}

One can check that the conditions $(\omega_r|\beta_l)=1$ and $(\rho|\beta_l)=l$ are satisfied for $1\le l \le r=h^{\vee}-1$.

\subsection*{Type $D_r$ ($r\ge 4$)}
We assume that $a\in \{1, r-1, r\}$. For $1\le l \le r-2$, we define $\beta_{h^{\vee}-l}$ as follows :  
\begin{itemize}

\item $\beta_{h^{\vee}-1}=\theta=\alpha_1+\alpha_{r-1}+\alpha_{r}+2\sum_{j=2}^{r-2}\alpha_j$
\item $\beta_{h^{\vee}-2}= s_2 \beta_{h^{\vee}-1}= \alpha_1+\alpha_2+\alpha_{r-1}+\alpha_{r}+2\sum_{j=3}^{r-2}\alpha_j$
\item $\cdots$
\item $\beta_{h^{\vee}-(r-3)}= s_{r-3} \beta_{h^{\vee}-(r-3)}= \sum_{j=1}^{r}\alpha_j+\alpha_{r-2}$
\item $\beta_{h^{\vee}-(r-2)}=s_{r-2}\beta_{h^{\vee}-(r-3)}=\sum_{j=1}^{r}\alpha_j$
\end{itemize}

Note that the conditions $(\omega_a|\beta_l)=1$ and $(\rho|\beta_l)=l$ are satisfied for $1\le l \le r-2$.
To define the next term $\beta_{h^{\vee}-(r-1)}=\beta_{r-1}$, choose the vertex $j\in \{r-1,r\}$ such that $j\neq a$ and let
$\beta_{r-1}=\beta_{h^{\vee}-(r-2)}-\alpha_{j}$. Since $\Pi- \{\alpha_j\}$ forms a simple system of type $A_{r-1}$ with the highest root $\beta_{r-1}$, we can now use the construction for type $A_{r-1}$ to define the rest of the terms $\beta_{r-2},\cdots, \beta_{1}=\alpha_a$. 

\subsection*{Type $E_6$}
Let $a\in \{1, 5\}$. For $1\le l \le 4$, we define $\beta_{h^{\vee}-l}$ as follows :  
\begin{itemize}

\item $\beta_{h^{\vee}-1}=\theta=\alpha _1+2 \alpha _2+3 \alpha _3+2 \alpha _4+\alpha _5+2 \alpha _6$
\item $\beta_{h^{\vee}-2}= s_6 \beta_{h^{\vee}-1}= \alpha _1+2 \alpha _2+3 \alpha _3+2 \alpha _4+\alpha _5+ \alpha _6$
\item $\beta_{h^{\vee}-3}= s_3 \beta_{h^{\vee}-2}= \alpha _1+2 \alpha _2+2 \alpha _3+2 \alpha _4+\alpha _5+ \alpha _6$
\item $\beta_{h^{\vee}-4}=  
\begin{cases} 
s_4 \beta_{h^{\vee}-3}=\alpha _1+2 \alpha _2+2 \alpha _3+\alpha _4+\alpha _5+\alpha _6 & \text{if $a=1$}\\
s_2 \beta_{h^{\vee}-3}=\alpha _1+\alpha _2+2 \alpha _3+2 \alpha _4+\alpha _5+\alpha _6 & \text{if $a=5$} 
\end{cases}$
\end{itemize}

Note that the conditions $(\omega_a|\beta_l)=1$ and $(\rho|\beta_l)=l$ are satisfied for $1\le l \le 4$. In order to define the next term $\beta_{h^{\vee}-5}=\beta_{7}$, choose the vertex $j\in \{1,5\}$ such that $j\neq a$ and let
$$
\beta_{h^{\vee}-5}= s_j \beta_{h^{\vee}-4}=
\begin{cases} 
\alpha _1+2 \alpha _2+2 \alpha _3+\alpha _4+\alpha _6 & \text{if $a=1$}\\
\alpha _2+2 \alpha _3+2 \alpha _4+\alpha _5+\alpha _6 & \text{if $a=5$} 
\end{cases}.
$$
Since $\Pi- \{\alpha_j\}$ forms a simple system of type $D_5$ with the highest root $\beta_{h^{\vee}-5}=\beta_{7}$
, we can now use the construction for type $D_5$ to define the rest of the sequence $\beta_{6},\cdots, \beta_{1}=\alpha_a$. 

\subsection*{Type $E_7$}
We have only one vertex $a=6$. For $1\le l \le r$, we define $\beta_{h^{\vee}-l}$ by
\begin{itemize}

\item $\beta_{h^{\vee}-1} = \theta= 2 \alpha _1+3 \alpha _2+4 \alpha _3+3 \alpha _4+2 \alpha _5+\alpha _6+2 \alpha _7$
\item $\beta_{h^{\vee}-2} = s_1 \beta_{h^{\vee}-1} = \alpha _1+3 \alpha _2+4 \alpha _3+3 \alpha _4+2 \alpha _5+\alpha _6+2 \alpha _7$
\item $\beta_{h^{\vee}-3} = s_2 \beta_{h^{\vee}-2} =\alpha _1+2 \alpha _2+4 \alpha _3+3 \alpha _4+2 \alpha _5+\alpha _6+2 \alpha _7$
\item $\beta_{h^{\vee}-4} = s_3 \beta_{h^{\vee}-3} = \alpha _1+2 \alpha _2+3 \alpha _3+3 \alpha _4+2 \alpha _5+\alpha _6+2 \alpha _7$
\item $\beta_{h^{\vee}-5} = s_4 \beta_{h^{\vee}-4} = \alpha _1+2 \alpha _2+3 \alpha _3+2 \alpha _4+2 \alpha _5+\alpha _6+2 \alpha _7$
\item $\beta_{h^{\vee}-6} = s_7 \beta_{h^{\vee}-5} = \alpha _1+2 \alpha _2+3 \alpha _3+2 \alpha _4+2 \alpha _5+\alpha _6+\alpha _7$
\item $\beta_{h^{\vee}-7} = s_3 \beta_{h^{\vee}-6} = \alpha _1+2 \alpha _2+2 \alpha _3+2 \alpha _4+2 \alpha _5+\alpha _6+\alpha _7$
\end{itemize}

We further define
\begin{itemize}

\item $\beta_{h^{\vee}-8} = s_2 \beta_{h^{\vee}-7} = \alpha _1+\alpha _2+2 \alpha _3+2 \alpha _4+2 \alpha _5+\alpha _6+\alpha _7$
\item $\beta_{h^{\vee}-9} = s_1 \beta_{h^{\vee}-8} = \alpha _2+2 \alpha _3+2 \alpha _4+2 \alpha _5+\alpha _6+\alpha _7$
\end{itemize}
One can see that the conditions $(\omega_a|\beta_l)=1$ and $(\rho|\beta_l)=l$ are satisfied for $9\le l \le 17$.
Since $\Pi- \{\alpha_1\}$ forms a simple system of type $D_6$ with the highest root $\beta_{h^{\vee}-9}=\beta_{9}$, we now imitate the construction of type $D_6$ to define the rest of the sequence $\beta_{8},\cdots, \beta_{1}=\alpha_a$. 

\section{Solutions of level $k$ restricted $Q$-systems}\label{appQdec}
Let $(a,m)\in H$. We describe the decomposition (\ref{Qdecomp}) of  $\operatorname{res} W^{(a)}_{m}$ into irreducibles for all classical types. See \cite[Appendix]{MR1745263} for a reference. In theses cases, $Z(a, m,\omega)$ is 0 or 1 for each $\omega \in P_{+}$. Hence it is enough to describe the set $\Omega^{(a)}_{m}$
of weights defined by 
$$\Omega^{(a)}_{m}:=\{\omega\in P_{+}:Z(a, m,\omega)=1\}.$$
We will give it in a form by which we can determine the coefficient of $\hat{\omega}_0$ easily when we extend $\omega\in \Omega^{(a)}_{m}$  to an affine weight of level $k$.

\subsection*{Type $A_r$}
For $1\le a \le r$,
\begin{equation}\label{Aomega1}
\Omega^{(a)}_{m}=\{m\omega_{a}\}.
\end{equation}

For $\omega\in \Omega^{(a)}_{m}$ in (\ref{Aomega1}), its level $k$ affinization $\hat{\omega}$ is given by
$$
\hat{\omega}=(k-m)\hat{\omega}_0+m\hat{\omega}_{a}
$$ 
and it is clear that $\hat{\omega}\in \hat{P}_{+}^{k}$ when $(a,m)\in H_{k}$.

\subsection*{Type $B_r$}
For $a\in I$ even,
\begin{equation}
\omega\in \Omega^{(a)}_{m}
\iff
\begin{cases}
\omega=k_{a}\omega_{a} + k_{a-2}\omega_{a- 2} + \cdots + k_2\omega_2\\ 
k_a+t_a(k_{a-2} + \cdots + k_2+k_{0})  = m\\
k_a,k_{a-2},\cdots, k_2,k_0 \in \mathbb{Z}_{\ge 0} \\
\end{cases}.\label{Bomega1}
\end{equation}
For $\omega\in \Omega^{(a)}_{m}$ in (\ref{Bomega1}), its level $k$ affinization $\hat{\omega}$ is given by
$$
\hat{\omega}=k_{a}\hat{\omega}_{a} + k_{a-2}\hat{\omega}_{a- 2} + \cdots + k_2\hat{\omega}_2+\hat{k}_0\hat{\omega}_0
$$ 
where
\begin{equation}\label{BomegaEven0}
\hat{k}_0=
\begin{cases} 
k-2m+2k_0 & \text{if $1\le a\le r-1$}\\
k-m+2k_0 & \text{if $a=r$} 
\end{cases}.
\end{equation}

For $a\in I$ odd,
\begin{equation}
\omega\in \Omega^{(a)}_{m}
\iff
\begin{cases}
\omega=k_{a}\omega_{a} + k_{a-2}\omega_{a- 2} + \cdots + k_1\omega_1\\
k_a+t_a(k_{a-2}+\cdots+k_1)= m\\
k_a,k_{a-2},\cdots, k_1 \in \mathbb{Z}_{\ge 0} \\
\end{cases}.\label{Bomega2}
\end{equation}
For $\omega\in \Omega^{(a)}_{m}$ in (\ref{Bomega2}), we get
$$
\hat{\omega}=k_{a}\hat{\omega}_{a} + k_{a-2}\hat{\omega}_{a- 2} + \cdots + k_1\hat{\omega}_1+\hat{k}_0\hat{\omega}_0
$$ 
where
\begin{equation}\label{BomegaOdd0}
\hat{k}_0=
\begin{cases} 
k-2m+k_1 & \text{if $1\le a\le r-1$}\\
k-m+k_1 & \text{if $a=r$} 
\end{cases}.
\end{equation}

\subsection*{Type $C_r$}
For $1 \le a \le r-1$,
\begin{equation}
\omega\in \Omega^{(a)}_{m}
\iff
\begin{cases}
\omega=k_{a}\omega_{a} + k_{a-1}\omega_{a-1} + \cdots + k_1\omega_1\\
k_a+k_{a-1}+\cdots+k_1+k_0= m\\
k_b \equiv m\delta_{a, b} \pmod 2 \\
k_a,k_{a-1},\cdots, k_1,k_0 \in \mathbb{Z}_{\ge 0} \\
\end{cases}.\label{Comega}
\end{equation}
For $\omega\in \Omega^{(a)}_{m}$ in (\ref{Comega}), the corresponding affine weight is
$$
\hat{\omega}=k_{a}\hat{\omega}_{a} + k_{a-1}\hat{\omega}_{a-1} + \cdots + k_1\hat{\omega}_1+\hat{k}_0\hat{\omega}_0
$$ 
where
\begin{equation}\label{Comega0}
\hat{k}_0=k-m+k_0.
\end{equation}

For $a = r$,
$$
\Omega^{(a)}_{m}=\{m\omega_{a}\}
$$

\subsection*{Type $D_r$}
For even $a$ such that $2\le a \le r-2$, 
\begin{equation}
\omega\in \Omega^{(a)}_{m}
\iff
\begin{cases}
\omega=k_{a}\omega_{a} + k_{a-2}\omega_{a- 2} + \cdots + k_2\omega_2\\
k_a+k_{a-2}+\cdots+k_2+k_0 = m\\
k_a,k_{a-2},\cdots, k_2,k_0 \in \mathbb{Z}_{\ge 0} \\
\end{cases}.\label{Domega1}
\end{equation}
For $\omega\in \Omega^{(a)}_{m}$ in (\ref{Domega1}), the level $k$ affinization is
$$
\hat{\omega}=k_{a}\hat{\omega}_{a} + k_{a-2}\hat{\omega}_{a- 2} + \cdots + k_1\hat{\omega}_1+\hat{k}_0\hat{\omega}_0
$$
where
$$
\hat{k}_0=k-2m+2k_0.
$$

For odd $a$ such that $1\le a \le r-2$,
\begin{equation}
\omega\in \Omega^{(a)}_{m}
\iff
\begin{cases}
\omega=k_{a}\omega_{a} + k_{a-2}\omega_{a- 2} + \cdots + k_1\omega_1\\
k_a+k_{a-2}+\cdots+k_1+k_0 = m\\
k_a,k_{a-2},\cdots, k_1,k_0 \in \mathbb{Z}_{\ge 0} \\
\end{cases}.\label{Domega2}
\end{equation}
For $\omega\in \Omega^{(a)}_{m}$ in (\ref{Domega2}), we get
$$
\hat{\omega}=k_{a}\hat{\omega}_{a} + k_{a-2}\hat{\omega}_{a- 2} + \cdots + k_1\hat{\omega}_1+\hat{k}_0\hat{\omega}_0
$$
where
$$
\hat{k}_0=k-2m+k_0+k_1.
$$

For $a = r-1$ and $r$, we have $\Omega^{(a)}_{m}=\{m\omega_{a}\}$.

\section{Proof of Lemma \ref{centrosym} for type $B_r$}\label{proofB}
Let $\hat{\Omega}^{(a)}_m=\{\hat{\omega}\in \hat{P}^k :\omega\in \Omega^{(a)}_m\}$. For the rest of this section, we will denote the element $k_a\hat{\omega}_{a}+k_{a-2}\hat{\omega}_{a-2}+\cdots + k_{2} \hat{\omega}_{2}+\hat{k}_{0} \hat{\omega}_{0}\in \hat{P}^k$ by $(k_a,k_{a-2},\cdots, k_2, \hat{k}_0)$ when $a$ is even and $k_{a}\hat{\omega}_{a}+k_{a-2}\hat{\omega}_{a-2}+\cdots + k_{1} \hat{\omega}_{1} +\hat{k}_{0} \hat{\omega}_{0}\in \hat{P}^k$ by $(k_a,k_{a-2},\cdots, k_1, \hat{k}_0)$ when $a$ is odd. Recall (\ref{pi}) the projection $\pi:\hat{P}\to P$ defined by
$$
\pi(\sum_{i=0}^{r}\lambda_{i}\hat{\omega}_{i})=\sum_{i=1}^{r}\lambda_{i}\omega_{i}.
$$

\subsection*{The case of the vertex $a=r$}
\begin{prop}
Let $r$ be even and $a=r$. We have
$$\beta_k(\operatorname{res} W_{k+1}^{(a)})=\beta_k(\operatorname{res} W_{k-1}^{(a)}).$$ 
\end{prop}
\begin{proof}
Note that $\hat{\Omega}^{(a)}_{k-1}\subseteq \hat{\Omega}^{(a)}_{k+1}$ and 
$$
\hat{\Omega}^{(a)}_{k+1} \setminus \hat{\Omega}^{(a)}_{k-1}= \left\{(k_a,k_{a-2},\cdots, k_2, \hat{k}_{0})\in \hat{P}^k\mid
\begin{array}{ll}
k_a+2(k_{a-2}+\cdots+k_2)=k+1 \\
k_a,k_{a-2},\cdots, k_2 \in \mathbb{Z}_{\geq 0} \\
\end{array}
\right\}.
$$
If $\hat{\omega}=(k_a,k_{a-2},\cdots, k_2, \hat{k}_0)\in \hat{\Omega}^{(a)}_{k+1} \setminus \hat{\Omega}^{(a)}_{k-1}$, then $\hat{k}_0=-1$ by (\ref{BomegaEven0}) and thus $\beta_{k}(V_{\pi(\hat\omega)})=0$ since $s_0\cdot \hat\omega=\hat\omega$. Hence $\beta_k(\operatorname{res} W_{k+1}^{(a)})=\beta_k(\operatorname{res} W_{k-1}^{(a)})$. 
\end{proof}

\begin{prop}
Let $r$ be odd and $a=r$. We have
$$\beta_k(\operatorname{res} W_{k+1}^{(a)})=V_{k\tau_{a}}\beta_k(\operatorname{res}W_{k-1}^{(a)})^{*}.$$
\end{prop}
\begin{proof}
For any $\omega=(k_a,k_{a-2},\cdots, k_1,\hat{k}_0)\in \hat{\Omega}^{(a)}_{k+1}$ with $k_1=0$, we get $\hat{k}_0=-1$ by (\ref{BomegaOdd0}) and thus $\beta_{k}(V_{\pi(\hat\omega)})=0$.

Let
$$
(\hat{\Omega}^{(a)}_{k+1})'= \{(k_a,k_{a-2},\cdots, k_1, \hat{k}_0)\in \hat{\Omega}^{(a)}_{k+1} \mid k_1\ge 1\}.$$
Let us define a map from $(\hat{\Omega}^{(a)}_{k+1})'$ to $\hat{\Omega}^{(a)}_{k-1}$ by 
\begin{equation}\label{s1110}
(k_a,k_{a-2},\cdots, k_1, \hat{k}_0) \mapsto (k_a,k_{a-2},\cdots, \hat{k}_0, k_1).
\end{equation}
For $(k_a,k_{a-2},\cdots, k_1, \hat{k}_0)\in (\hat{\Omega}^{(a)}_{k+1})'$, we get $\hat{k}_0=k_1-1\ge 0$ by (\ref{BomegaOdd0}). This shows that $(k_a,k_{a-2},\cdots, \hat{k}_0, k_1)\in \hat{\Omega}^{(a)}_{k-1}$ and thus the map (\ref{s1110}) is well-defined. It is clear that this is injective.

Conversely, any element $(k_a,k_{a-2},\cdots, k_1, \hat{k}_0)\in \hat{\Omega}^{(a)}_{k-1}$ satisfies $\hat{k}_0=k_1+1\ge 1$ again by (\ref{BomegaOdd0}) and it proves that (\ref{s1110}) is surjective and thus bijective. This proves our proposition.
\end{proof}

\subsection*{The case of the vertices $1\le a \le r-1$ when $k$ is odd}
Let $s=\frac{k-1}{2}$. 
\begin{prop}
If $a$ is even and $1\le a \le r-1$, then 
$$\beta_k(\operatorname{res} W^{(a)}_{s})=\beta_k(\operatorname{res} W^{(a)}_{s+1}).$$
\end{prop}
\begin{proof}
Note that $\hat{\Omega}^{(a)}_{s}\subseteq \hat{\Omega}^{(a)}_{s+1}$ and 
$$
\hat{\Omega}^{(a)}_{s+1} \setminus \hat{\Omega}^{(a)}_{s}= \left\{(k_a,k_{a-2},\cdots, k_2, \hat{k}_0)\in \hat{P}^k\mid
\begin{array}{ll}
k_a+k_{a-2}+\cdots+k_2= s+1 \\
k_a,k_{a-2},\cdots, k_2 \in \mathbb{Z}_{\geq 0} \\
\end{array}
\right\}.
$$
If $\hat{\omega}=(k_a,k_{a-2},\cdots, k_2, \hat{k}_0)\in \hat{\Omega}^{(a)}_{s+1} \setminus \hat{\Omega}^{(a)}_{s}$, then $\hat{k}_0=-1$. 
So for any $\hat{\omega}\in \hat{\Omega}^{(a)}_{s+1} \setminus \hat{\Omega}^{(a)}_{s}$, $\beta_{k}(V_{\pi(\hat\omega)})=0$ since $s_0\cdot \hat\omega=\hat\omega$.
Thus $\beta_k(\operatorname{res} W^{(a)}_{s+1})=\beta_k(\operatorname{res}W^{(a)}_{s})$.
\end{proof}

\begin{prop}
If $a$ is odd and $1\le a \le r-1$, then we have
$$\beta_k(\operatorname{res}W_{s+1}^{(a)})=V_{k\tau_{a}}\beta_k(\operatorname{res}W_{s}^{(a)})^{*}.$$
\end{prop}
\begin{proof}
Let $\hat{\omega}=(k_a,k_{a-2},\cdots, k_1, \hat{k}_0)\in \hat{\Omega}^{(a)}_{s+1}$. If $k_1=0$, then $\hat{k}_0=-1$ by (\ref{BomegaOdd0}) and thus $V_{\hat\omega}=0$ since $s_0\cdot \hat\omega=\hat\omega$.

Let $$
(\hat{\Omega}^{(a)}_{s+1})'= \{(k_a,k_{a-2},\cdots, k_1, \hat{k}_0)\in \hat{\Omega}^{(a)}_{s+1} \mid k_1\ge 1\}.$$

Let us construct a bijection between $(\hat{\Omega}^{(a)}_{s+1})'$ and $\hat{\Omega}^{(a)}_{s}$. Define a map from $(\hat{\Omega}^{(a)}_{s+1})'$ to $\hat{\Omega}^{(a)}_{s}$ by
\begin{equation}\label{map010}
(k_a,k_{a-2},\cdots, k_1, \hat{k}_0) \mapsto (k_a,k_{a-2},\cdots, \hat{k}_0, k_1).
\end{equation}
To see that the map is well-defined, note that if $(k_a,k_{a-2},\cdots, k_1, \hat{k}_0)\in (\hat{\Omega}^{(a)}_{s+1})'$, then $\hat{k}_0=k_1-1\ge 0$ by (\ref{BomegaOdd0}). Since 
$$
k_a+k_{a-2}+\cdots+k_3+ \hat{k}_0=k_a+k_{a-2}+\cdots+k_3+ (k_{1}-1)=s,
$$
we have $(k_a,k_{a-2},\cdots, \hat{k}_0, k_1)\in \hat{\Omega}^{(a)}_{s}$. The map (\ref{map010}) is clearly injective.

Conversely, any element $(k_a,k_{a-2},\cdots, k_1, \hat{k}_0)\in \hat{\Omega}^{(a)}_{s}$ satisfies the condition $\hat{k}_0=k_1+1\ge 1$ which shows that (\ref{map010}) is surjective. We thus have proved that (\ref{map010}) is a bijection between $(\hat{\Omega}^{(a)}_{s+1})'$ and $\hat{\Omega}^{(a)}_{s}$. This proves our assertion.
\end{proof}

\subsection*{The case of the vertices $1\le a \le r-1$ when $k$ is even}
Let $s=\frac{k}{2}$. 

\begin{lemma}\label{koddm4}
Let $a$ be even and $1\le a \le r-1$. If $\hat\omega=(k_a,k_{a-2},\cdots, k_2,\hat{k}_0)\in \hat{P}^k$ satisfies $k_2=0$ and $\hat{k}_0=-2$, then $\beta_{k}(V_{\pi(\hat\omega)})=0$.
\end{lemma}
\begin{proof}
From $(s_0s_2s_0)\cdot (k_a,k_{a-2},\cdots, 0, -2)=(k_a,k_{a-2},\cdots, 0, -2)$, we can deduce that $\beta_{k}(V_{\pi(\hat\omega)})=0$.
\end{proof}

\begin{prop}
If $a$ is even and $1\le a \le r-1$, then $\beta_k(\operatorname{res}W^{(a)}_{s+1})=\beta_k(\operatorname{res}W^{(a)}_{s-1})$.
\end{prop}
\begin{proof}
Recall that
$$
\hat{\Omega}^{(a)}_{s-1}= \left\{(k_a,k_{a-2},\cdots, k_2, \hat{k}_0)\in \hat{P}^k\mid
\begin{array}{ll}
k_a+k_{a-2}+\cdots+k_2 \le s-1\\
k_a,k_{a-2},\cdots, k_2 \in \mathbb{Z}_{\geq 0} \\
\end{array}
\right\}
$$
and 
$$
\hat{\Omega}^{(a)}_{s+1}= \left\{(k_a,k_{a-2},\cdots, k_2, \hat{k}_0)\in \hat{P}^k\mid
\begin{array}{ll}
k_a+k_{a-2}+\cdots+k_2 \le s+1\\
k_a,k_{a-2},\cdots, k_2 \in \mathbb{Z}_{\geq 0} \\
\end{array}
\right\}.
$$

Let us define three disjoint subsets $R,S$ and $T$ of $\hat{\Omega}^{(a)}_{s+1}$ by
\begin{align}
&R=\{(k_a,k_{a-2},\cdots, k_2, \hat{k}_0)\in \hat{\Omega}^{(a)}_{s+1}:k_a+k_{a-2}+\cdots+k_2=s+1, k_2 = 0\}, \notag\\
&S=\{(k_a,k_{a-2},\cdots, k_2, \hat{k}_0)\in \hat{\Omega}^{(a)}_{s+1}:k_a+k_{a-2}+\cdots+k_2=s+1, k_2 \ge 1\}, \notag\\
&T=\{(k_a,k_{a-2},\cdots, k_2, \hat{k}_0)\in \hat{\Omega}^{(a)}_{s+1}:k_a+k_{a-2}+\cdots+k_2=s\}. \notag
\end{align}

For $\hat{\omega}\in R$, we get $\hat{k}_0=-2$. Thus $\beta_{k}(V_{\pi(\hat\omega)})=0$ by Lemma \ref{koddm4} and so $\sum_{\hat{\omega}\in R} \beta_{k}(V_{\pi(\hat\omega)})=0$. We now want to prove $\sum_{\hat{\omega}\in S\cup T}\beta_{k}(V_{\pi(\hat\omega)})=0$. By (\ref{BomegaEven0}), if $(k_a,k_{a-2},\cdots, k_2, \hat{k}_0)\in S$, then $\hat{k}_0=-2$ and if $(k_a,k_{a-2},\cdots, k_2,\hat{k}_0)\in T$, then $\hat{k}_0=0$. We have a bijection between $S$ and $T$ since $$s_0\cdot (k_a,k_{a-2},\cdots, k_2, -2)=(k_a,k_{a-2},\cdots, k_2-1, 0).$$
By (\ref{Valcove}), $\sum_{\hat{\omega}\omega\in S\cup T} \beta_{k}(V_{\pi(\hat\omega)})=0$. Consequently, $$\beta(\operatorname{res} W^{(a)}_{s+1})=\sum_{\hat{\omega} \in \hat{\Omega}^{(a)}_{s+1}} \beta_{k}(V_{\pi(\hat\omega)})=\sum_{\hat{\omega} \in \hat{\Omega}^{(a)}_{s+1}\setminus (R\cup S\cup T)} \beta_{k}(V_{\pi(\hat\omega)}).$$

From $\hat{\Omega}^{(a)}_{s-1}=\hat{\Omega}^{(a)}_{s+1}\setminus (R\cup S\cup T)$, we obtain $\beta_k(\operatorname{res} W^{(a)}_{s+1})=\beta_k(\operatorname{res} W^{(a)}_{s-1})$.
\end{proof}

\begin{lemma}\label{koddm5}
Let $a$ be odd and $1\le a \le r-1$. If $\hat\omega=(k_a,k_{a-2},\cdots,1, -1)\in \hat{P}^k$ or $\hat\omega=(k_a,k_{a-2},\cdots, 0, -2)\in \hat{P}^k$, then $\beta_{k}(V_{\pi(\hat\omega)})=0$.
\end{lemma}

\begin{proof}
It suffices to check that
$$s_0\cdot (k_a,k_{a-2},\cdots,1, -1)=(k_a,k_{a-2},\cdots,1, -1)$$ 
and
$$(s_0s_2s_0)\cdot (k_a,k_{a-2},\cdots, 0, -2)=(k_a,k_{a-2},\cdots, 0, -2).$$
\end{proof}

\begin{prop}
If  $a$ is odd and $1\le a \le r-1$, then $\beta_k(\operatorname{res} W_{s+1}^{(a)})=V_{k\tau_{a}}\beta_k(\operatorname{res}W_{s-1}^{(a)})^{*}$.
\end{prop}
\begin{proof}
For any $\hat{\omega}=(k_a,k_{a-2},\cdots, k_1, \hat{k}_0)\in \hat{\Omega}^{(a)}_{s+1}$ with $k_1=0$ or $k_1=1$, $\beta_{k}(V_{\pi(\hat\omega)})=0$ by Lemma \ref{koddm5}. Let $(\hat{\Omega}^{(a)}_{s+1})'= \{(k_a,k_{a-2},\cdots, k_1, \hat{k}_0)\in \hat{\Omega}^{(a)}_{s+1} \mid k_1\ge 2\}
$.
Then we can write $\beta_k(\operatorname{res}(W^{(a)}_{s+1})=\sum_{\hat{\omega}\in (\hat{\Omega}^{(a)}_{s+1})'} \beta_{k}(V_{\pi(\hat\omega)})$.

Let us define a map from $(\hat{\Omega}^{(a)}_{s+1})'$ to $\hat{\Omega}^{(a)}_{s-1}$ by 
\begin{equation}\label{s111}
(k_a,k_{a-2},\cdots, k_1, \hat{k}_0) \mapsto (k_a,k_{a-2},\cdots, \hat{k}_0, k_1).
\end{equation}
For $(k_a,k_{a-2},\cdots, k_1, \hat{k}_0)\in (\hat{\Omega}^{(a)}_{s+1})'$, $\hat{k}_0=k_1-2\ge 0$. It shows that $$(k_a,k_{a-2},\cdots, \hat{k}_0, k_1)\in \hat{\Omega}^{(a)}_{s-1}$$ and thus the map (\ref{s111}) is well-defined. It is clear that this is injective.

Conversely, any element $(k_a,k_{a-2},\cdots, k_1, \hat{k}_0)\in \hat{\Omega}^{(a)}_{s-1}$ satisfies $\hat{k}_0=k_1+2\ge 2$. It proves that (\ref{s111}) is surjective and thus bijective. This finishes our proof of the proposition.
\end{proof}

\section{Proof of Lemma \ref{centrosym} for type $C_r$}\label{proofC}
\subsection*{The case of the vertices $1\le a\le r-1$}
\begin{aprop}
If $1\le a\le r-1$, then $\beta_k(\operatorname{res} W^{(a)}_{k-1})=\beta_k(\operatorname{res} W^{(a)}_{k+1})$.
\end{aprop}

\begin{proof}
Note that $\hat{\Omega}^{(a)}_{k-1}\subseteq \hat{\Omega}^{(a)}_{k+1}$ and 
$$
\hat{\Omega}^{(a)}_{k+1} \setminus \hat{\Omega}^{(a)}_{k-1}= \left\{(k_a,k_{a-1},\cdots, k_1, \hat{k}_0)\in \hat{P}^k\mid
\begin{array}{ll}
k_a+k_{a-1}+\cdots+k_1=k+1 \\
k_a,k_{a-1},\cdots, k_1 \in \mathbb{Z}_{\geq 0} \\
\end{array}
\right\}.
$$
If $\hat{\omega}=(k_a,k_{a-1},\cdots, k_1, \hat{k}_0)\in \hat{\Omega}^{(a)}_{k+1} \setminus \hat{\Omega}^{(a)}_{k-1}$, then $\hat{k}_0=-1$ by (\ref{Comega0}). So for any $\hat{\omega}\in \hat{\Omega}^{(a)}_{k+1} \setminus \hat{\Omega}^{(a)}_{k-1}$, $\beta_{k}(V_{\pi(\hat\omega)})=0$ since $s_0\cdot \hat\omega=\hat\omega$.
 Thus $\beta_k(\operatorname{res} W^{(a)}_{k+1})=\beta_k(\operatorname{res} W^{(a)}_{k-1})$.
\end{proof}

\subsection*{The case of the vertex $a=r$}
It follows from Theorem \ref{simplepconj}.